\documentclass[11pt,reqno]{amsart}

\title[]{On the Dirichlet Fractional Laplacian and Applications to the SQG Equation on Bounded Domains}

\author[E. Abdo]{Elie Abdo}
\address[E. Abdo]
{	Department of Mathematics \\
     University of California  \\
	Santa Barbara, CA 93106-3080, USA.} \email{elieabdo@ucsb.edu}

\author[Q. Lin]{Quyuan Lin}
\address[Q. Lin]
{	School of Mathematical and Statistical Sciences \\
Clemson University\\
Clemson, SC 29634, USA.} \email{quyuanl@clemson.edu}

\numberwithin{equation}{section}

\usepackage[margin=1in]{geometry}
\usepackage{amsmath, amsthm, amssymb}
\usepackage{times}
\usepackage{comment}
\usepackage{color}
\usepackage[colorlinks=true, pdfstartview=FitV, linkcolor=blue, citecolor=blue, urlcolor=blue]{hyperref}
\usepackage{hyperref}
\newcommand{\pa}{\partial}
\newcommand{\p}{\partial}
\newcommand{\la}{\label}
\newcommand{\fr}{\frac}
\newcommand{\na}{\nabla}
\newcommand{\be}{\begin{equation}}
\newcommand{\ee}{\end{equation}}
\newcommand{\ba}{\begin{array}{l}}
\newcommand{\ea}{\end{array}}

\newtheorem{thm}{Theorem}[section]
\newtheorem{prop}[thm]{Proposition}

\newtheorem{cor}[thm]{Corollary}
\newtheorem{rem}{Remark}

\newcommand{\beg}{\begin}

\newcommand{\D}{\Delta}
\renewcommand{\l}{\Lambda_D}

\usepackage{cite}
\usepackage{mathtools}
\mathtoolsset{showonlyrefs}

\usepackage{xcolor}
\usepackage{graphicx}
\usepackage[]{amsfonts}
\usepackage{amsthm}
\usepackage{amsmath}
\usepackage {amssymb}
\usepackage{bm}
\usepackage{mathrsfs}
\usepackage{setspace}
\usepackage{amsmath}

\usepackage{MnSymbol,wasysym}

\newcommand{\R}{\mathbb R}

\def\RR{{\mathbb R}}

\def\NN{\mathbb N}

\begin{document}
\begin{abstract} We investigate new properties of the fractional Dirichlet Laplacian on smooth bounded domains and establish fractional product estimates and nonlinear Poincar\'e inequalities. We also use these tools to study the long-time dynamics of the surface quasi-geostrophic equation forced by some given time-independent body forces in the presence of physical boundaries and prove the existence of a finite-dimensional global attractor.  
\end{abstract} 


\maketitle



\tableofcontents

\section{Introduction}

The fractional Laplacian is a nonlocal operator that appears in many drift-diffusion partial differential equations (PDE) studied in the mathematical literature, including but not limited to the surface quasi-geostrophic (SQG) \cite{caffarelli2010drift, constantin2001critical, constantin2016critical,  constantin2020estimates, constantin2018inviscid, constantin2023global, constantin1994formation, constantin2018local, constantin2015long, constantin2012nonlinear, constantin1999behavior, constantin2008regularity, ignatova2019construction, zelati2016global, kiselev2007global, stokols2020holder}, incompressible porous media (IPM) \cite{castro2009incompressible}, fractional Boussinesq \cite{yang2014global, wu2014well, wu20182d, yang20183d}, fractional magnetohydrodynamics (MHD) \cite{chae2015local, dai2020unique, dai2019class,  dong2018global, feng2024stability,  wu20182d}, and electroconvection models \cite{abdo2021long, abdo2024long, constantin2017some}. The behavior of this operator is majorly shaped by the geometry of the physical domain on which these equations are studied and by the type of boundary conditions imposed on the physical quantities whose evolutions are addressed.

In the absence of spatial boundaries (as in the cases of the whole spaces and periodic tori), the fractional powers of the Laplacian are defined as Fourier multipliers and have integral representation formulas with explicit kernels. Moreover, they commute with differential operators, their domains are identified with the usual fractional Sobolev spaces, and fractional product estimates are available in the classical Sobolev spaces. These properties wipe out major challenges in the analysis of PDE problems in which such operators are involved but fully break down when defined in the context of bounded domains with prescribed boundary conditions.

A few results were recently obtained in the literature and used to study nonlocal PDEs on bounded domains. Integro-differential representations with unexplicit kernels are derived in \cite{caffarelli2016fractional} in the case of homogeneous Dirichlet and Neumann boundary conditions. Moreover, nonlinear lower bounds, maximum principles, and commutator estimates are established in \cite{constantin2017remarks} in the case of Dirichlet boundary conditions.  

In this paper, we consider a two-dimensional bounded domain $\Omega$ with a smooth boundary, study new features fulfilled by the Dirichlet fractional Laplacian, and use them to address the global well-posedness and long-time dynamics of the forced subcritical SQG equation posed on $\Omega$. Denoting the homogeneous Dirichlet Laplacian by $-\Delta_D$ and its fractional powers by $\l^s$ where $\l = \sqrt{-\Delta_D}$, we successfully obtain the following results:
\begin{enumerate}
    \item We consider a family of nonlocal regularizers $J_{\epsilon}$ that are uniformly bounded on the domains of $\l^s$ and prove that they commute with $\l^s$ based on singular integral representation identities (Proposition \ref{prop:commute}). By making use of the Marcinkiewicz interpolation theorem, we also study the uniform-in-$\epsilon$ boundedness properties of the operators $\l^s J_{\epsilon}$ on $L^{p}$ spaces (Theorem \ref{thm:Lp}). These results serve as a major tool for constructing regularization schemes for nonlocal PDEs involving fractional powers of the Laplacian that align with the usual mollification arguments performed in the absence of spatial boundaries. 
    \item We establish several fractional product estimates to control $\l^s(fg)$ in $L^2$ depending on the regularity and boundary assumptions obeyed by $f$ and $g$ (Theorem \ref{thm:product-1}). We also derive nontrivial trilinear product estimates based on extension theorems and the properties of the fractional Laplacian on the whole space (Theorem \ref{thm:nonlinear}). These provide new tools to control the nonlinear aspects governing many nonlocal nonlinear PDEs.
    \item We prove a nonlinear Poincar\'e inequality on $L^p$ spaces that generalizes an analogous result obtained in \cite{constantin2014unique} on the periodic torus to the case of bounded domains for a wider range of $p$ (Theorem \ref{thm:poin}).  Our proof is based on the pointwise C\'ordoba-C\'ordoba inequality and a nontrivial analysis of the cases in which $|q|^\beta$ belongs to the domain of $\l^s$ provided that $q$ itself is there (Proposition \ref{prop:regularity}). These nonlinear lower bounds are useful for studying the long-time dynamics of solutions to nonlinear PDEs that are diffused by the Dirichlet fractional Laplacian.  
    \item Finally, we consider the two-dimensional subcritical SQG equation on $\Omega$ forced by time-independent body forces, and we address the existence of weak and strong unique solutions (Theorem~\ref{thm:sqg-1} and \ref{thm:sqg-2}). We also prove the existence of a finite-dimensional global attractor to this system (Theorem~\ref{thm:sqg-3} and \ref{thm:sqg-4}). The result is new and our analysis is majorly based on the previously listed tools.  
\end{enumerate}

This paper is organized as follows. In Section \ref{sec2}, we introduce the functional spaces, define the fractional powers of the Dirichlet Laplacian, list some of their properties, and recall some fractional interpolation inequalities that will be used in the sequel. 
In Section \ref{sec3}, we address the properties of smoothing regularizers on the domains of the fractional powers of the Laplacian. Section \ref{sec4} is dedicated to fractional bilinear and trilinear product estimates. In Section \ref{sec5}, we prove a nonlinear Poincar\'e inequality for the fractional powers of the Laplacian, adapted to the case of bounded domains with smooth boundaries. Finally, we apply our theory  to the subcritical forced SQG equation in Section \ref{sec6} and show the existence of weak and strong unique solutions and obtain the existence of a finite-dimensional global attractor. 

\section{Preliminaries} \la{sec2}

\subsection{Functional Setting.} Let $\Omega \subset \R^2$ be a bounded domain with smooth boundary. For $1 \le p \le \infty$, we denote by $L^p(\Omega)$ the Lebesgue spaces of measurable functions $f$ from $\Omega$ to $\R$ (or $\RR^2)$ such that 
\be 
\|f\|_{L^p} = \left(\int_{\Omega} |f(x)|^p dx\right)^{1/p} <\infty
\ee if $p \in [1, \infty)$ and 
\be 
\|f\|_{L^{\infty}} = {\mathrm{esssup}}_{\Omega}  |f| < \infty
\ee if $p = \infty$. The $L^2$ inner product is denoted by $(\cdot,\cdot)_{L^2}$. 

For $k \in \NN$, we denote by $H^k(\Omega)$ the Sobolev spaces of measurable functions $f$ from $\Omega$ to $\R$ (or $\RR^2)$ with weak derivatives of order $k$ such that  
\be 
\|f\|_{H^k}^2 = \sum\limits_{|\alpha| \le k} \|D^{\alpha}f\|_{L^2}^2 < \infty,
\ee and by $H_{0}^{1}(\Omega)$ the closure of $C_0^{\infty}(\Omega)$ in $H^1(\Omega)$. 

For a Banach space $(X, \|\cdot\|_{X})$ and $p\in [1,\infty]$, we consider the Lebesgue spaces $ L^p(0,T; X)$ of functions $f$  from $X$ to $\R$ (or $\RR^2)$ satisfying 
\be 
\int_{0}^{T} \|f\|_{X}^p dt  <\infty
\ee with the usual convention when $p = \infty$. 

\subsection{Fractional Powers of the Laplacian}  We denote by $\Delta_D$ the Laplacian operator with homogeneous Dirichlet boundary conditions. We note that $-\Delta_D$  is defined on $\mathcal{D}(-\D_D) = H^2(\Omega) \cap H_0^1(\Omega)$, and is positive and self-adjoint in $L^2(\Omega)$. Then there exists an orthonormal basis of $L^2(\Omega)$ consisting of eigenfunctions $\left\{w_j\right\}_{j=1}^{\infty} \subset H_0^1(\Omega)$ of $-\Delta_D$ satisfying
\be
-\Delta_D w_j = \lambda_j w_j
\ee
where the eigenvalues $\lambda_j$ obey $0 < \lambda_1 \leq ... \leq \lambda_j \le ...\rightarrow \infty$. 
For $s \in \R$, we define the fractional Laplacian operator of order $s$, denoted by $\l^s$, as
\be \la{maindef}
\l^s h= \sum_{j=1}^{\infty} \lambda_j^{\fr{s}{2}} (h, w_j)_{L^2} w_j,
\ee with domain 
\be 
\mathcal{D}(\l^s) = \left\{h  : \|\l^{s} h\|_{L^2}^2 := \sum\limits_{j=1}^\infty \lambda_j^{s}(h, w_j)_{L^2}^2 < \infty \right\}.
\ee 
In particular, when $s>0$ the space $\mathcal D(\l^{-s})$ is understood as the dual space of $\mathcal D(\l^{s})$. It is evident that  $\mathcal D(\l^{s_2}) \subset \mathcal D(\l^{s_1})$ provided that $s_1 \le s \le s_2$.
For $s \in [0,1]$, we identify the domains $\mathcal{D}(\l^s)$ with the usual Sobolev spaces as follows,
\be 
\mathcal{D}(\l^s) = \begin{cases} H^{s}(\Omega) \hspace{8cm} \mathrm{if} \; s \in [0, \fr{1}{2}), 
\\ H_{00}^{\fr{1}{2}} (\Omega) = \left\{h \in H_0^\fr{1}{2} (\Omega) : h/\sqrt{d(x)} \in L^2(\Omega)\right\} \hspace{2cm} \mathrm{if} \; s = \fr{1}{2},
\\H_{0}^{s} (\Omega) \hspace{8cm} \mathrm{if} \; s \in (\fr{1}{2},1], 
\end{cases}
\ee where $H_0^{s}(\Omega)$ is the Hilbert subspace of $H^s(\Omega)$ with vanishing boundary trace elements, and $d(x)$ is the distance to the boundary function.

We recall the identity 
\be 
\lambda^{\fr{s}{2}} = c_s \int_{0}^{\infty} t^{-1-\fr{s}{2}} (1-e^{-t\lambda}) dt
\ee 
that holds for $s \in (0,2)$, where $c_s$ is given by
\be 
1 = c_s \int_{0}^{\infty} t^{-1-\fr{s}{2}} (1-e^{-t}) dt.
\ee Using the latter, we  obtain the integral representation 
\be \la{intrep}
(\l^s f)(x) = c_s \int_{0}^{\infty} [f(x) - e^{t\Delta_D}f(x)]t^{-1 - \fr{s}{2}} dt
\ee for $f \in \mathcal{D}(\l^s)$ and $s \in (0,2)$. Here the heat operator $e^{t\Delta_D}$ is defined as
\be \label{eqn:etdelta}
(e^{t\Delta_D}f)(x) = \int_{\Omega} H_D(x, y, t) f(y) dy
\ee with kernel $H_D(x,y, t)$ given by 
\be \label{eqn:heat-kernel}
H_D(x,y,t) = \sum_{j=1}^{\infty} e^{-t\lambda_j} w_j(x) w_j(y).
\ee

\beg{prop} \cite{abdo2023long} The following identities hold:
\begin{enumerate}
\item[(i)] Let $\alpha, \beta, s \in \R$. For $f \in \mathcal{D}(\l^{\alpha}) \cap \mathcal{D}(\l^{\alpha - s}) $ and $g \in \mathcal{D}(\l^{\beta + s}) \cap \mathcal{D}(\l^{\beta})$, we have
\be 
(\l^{\alpha}f, \l^{\beta}g)_{L^2} = (\l^{\alpha - s}f, \l^{\beta + s}g)_{L^2}.
\ee
\item[(ii)] Let $\alpha, \beta \in \R$. For $f \in \mathcal{D}(\l^{\alpha + 1})$ and $g \in \mathcal{D}(\l^{\beta + 1})$, we have 
\be 
(\l^{\alpha + 1}f, \l^{\beta + 1}g)_{L^2} = (\na \l^{\alpha}f, \na \l^{\beta}g)_{L^2}.
\ee
\item[(iii)] Let $s \in (0,1)$. For $\psi \in \mathcal{D}(\l^{s})$, we have 
\be \la{prodfor1}
\|\l^{s} \psi\|_{L^2}^2
= \int_{\Omega} \int_{\Omega} (\psi(x) - \psi(y))^2 K_{s}(x,y) dxdy 
+ \int_{\Omega} \psi(x)^2 B_{s}(x) dx
\ee where the kernels $K_{s}$ and $B_{s}$ are given by  
\be 
K_s(x,y) := 2c_{2s} \int_{0}^{\infty} H(x,y,t) t^{-1-s} dt \quad \text{for }x \ne y,
\ee 
\be 
B_s(x) = c_{2s} \int_{0}^{\infty} \left[1 - e^{t\Delta}1(x)\right] t^{-1-s}dt \quad \text{for }x\in \Omega,
\ee and obey
\be \la{prodfor2}
0 \le K_{s}(x,y) \le \fr{C_s}{|x-y|^{2 + 2s}} \quad \text{for }x \ne y, 
\ee 
\be \la{prodfor3}
B_{s}(x) \ge 0 \quad \text{for }x\in \Omega.
\ee 
\end{enumerate}
\end{prop}

\subsection{Fractional Sobolev spaces and Brezis-Mironescu interpolation inequality.} For $s\in(0,1)$ and $p\in[1,\infty)$, the fractional Sobolev space $W^{s, p}(\Omega)$ is defined as
\be \label{def:frac-sob}
W^{s, p} (\Omega) = \left\{ v \in L^p(\Omega): \|v\|_{W^{s, p}}= \left(\|v\|_{L^p}^p + \int_{\Omega} \int_{\Omega} \fr{|v(x) - v(y)|^p}{|x-y|^{2+sp}} dxdy \right)^{\fr{1}{p}}<\infty \right\}.
\ee 
For a noninteger $s>1$, we write $s=m+\sigma$ where $m$ is a positive integer and $\sigma\in(0,1)$, and we define the space $W^{s, p}(\Omega)$ by
\be 
W^{s, p} (\Omega) = \left\{ v \in W^{m,p}(\Omega): \|v\|_{W^{s, p}}= \left(\|v\|_{W^{m,p}}^p + \sum\limits_{|\alpha|=m} \|D^\alpha v\|_{W^{\sigma,p}}^p \right)^{\fr{1}{p}}<\infty \right\}.
\ee 
For more details about fractional Sobolev spaces, we refer the readers to \cite{di2012hitchhiker}.

Let $1 \le p, p_1, p_2 \le \infty$ with $p_2 \ne 1$, and let $s, s_1, s_2$ be nonnegative real numbers such that $s_1 \leq s \le s_2$. Let $\theta \in (0,1)$ be such that $s = \theta s_1 + (1-\theta) s_2$ and $\fr{1}{p} = \fr{\theta}{p_1} + \fr{1-\theta}{p_2}$. Then there exists a positive universal constant $C$ such that the following interpolation inequality
\be 
\|f\|_{W^{s,p}} \le C\|f\|_{W^{s_1, p_1}}^{\theta}\|f\|_{W^{s_2, p_2}}^{1-\theta}
\ee holds for any $f \in W^{s_1, p_1}(\Omega) \cap W^{s_2, p_2} (\Omega)$. We refer the reader to \cite{brezis2018gagliardo} for a detailed proof.

\section{Regularizers} \la{sec3} For $\epsilon \in (0,1)$, we let $J_{\epsilon}$ be the spectrally regularizing operator defined in terms of the heat semigroup $e^{t\Delta_D}$ by 
\be 
J_{\epsilon} \theta(x)
= \frac{-1}{\ln \epsilon} \int_{\epsilon}^{\fr{1}{\epsilon}} \frac{e^{t\Delta_D}\theta(x)}{t} dt.
\ee The operator $J_{\epsilon}$ obeys the following properties:

\begin{prop} \label{regop}
Let $s$ be a real number and $\epsilon \in (0,1)$ be a small positive number. There exists a positive number $C$ depending only on $s$ such that
\be \label{prop32}
\|\l^s J_{\epsilon} \theta\|_{L^2} \le C\|\l^s\theta\|_{L^2},
\ee 
and 
\be 
\lim\limits_{\epsilon \to 0^+} \|\l^s (J_{\epsilon} \theta - \theta)\|_{L^2} = 0
\ee provided that $\theta \in \mathcal{D}(\l^{s})$. For $s\geq 0$ and $\beta\in \mathbb R$, it also holds that 
\begin{equation}
    \|\Lambda_D^{s+\beta} J_{\epsilon} \theta\|_{L^2} \le C \epsilon^{-\frac{s}2} \|\l^{\beta}\theta\|_{L^2}
\end{equation} for $\theta\in \mathcal D(\l^\beta)$.
\end{prop}

The proof of Proposition \ref{regop} follows the proof of Lemma in \cite{elie} and will be omitted here.

\beg{prop}\label{prop:commute}
{
Let $\epsilon \in (0,1)$, $s>0$, and $f \in \mathcal{D}(\l^s)$. Then 
\be 
\l^s J_{\epsilon} f (x) = J_{\epsilon} \l^s f(x)
\ee for a.e. $x \in \Omega$. In other words, the operators $\l^s$ and $J_{\epsilon}$ commute. 
}
\end{prop}

\beg{proof} We start with the case when $s\in (0,2)$. For $\eta \in (0,1)$, we define the truncated fractional Laplacian 
\be \label{eqn:trun-Lap}
(\l^s)_{\eta} f(x) = c_s \int_{\eta}^{\infty}  [f(x) - e^{t\Delta_D} f(x)] t^{-1-\frac{s}{2}} dt.
\ee We have 
\be 
\beg{aligned}
(\l^s)_{\eta} J_{\epsilon} f(x)
&= c_s \int_{\eta}^{\infty} [J_{\epsilon} f(x) - e^{t\Delta_D} J_{\epsilon}f(x)] t^{-1-\frac{s}{2}} dt 
\\&=-\frac{c_s}{\ln \epsilon} \int_{\eta}^{\infty} \left[\left(\int_{\epsilon}^{\frac{1}{\epsilon}} \frac{e^{\gamma \Delta_D}f(x)}{\gamma} d\gamma \right) - \left(\int_{\epsilon}^{\frac{1}{\epsilon}}\frac{e^{(t+\gamma)\Delta_D}f(x)}{\gamma} d\gamma \right)\right]t^{-1-\frac{s}{2}}dt
\\&= -\frac{c_s}{\ln \epsilon} \int_{\eta}^{\infty} \int_{\epsilon}^{\frac{1}{\epsilon}}e^{\gamma \Delta_D} [f(x) - e^{t\Delta_D} f(x)] t^{-1-\frac{s}{2}} \gamma^{-1} d\gamma dt
\\&= -\frac{1}{\ln \epsilon} \int_{\epsilon}^{\frac{1}{\epsilon}} \frac{e^{\gamma \Delta_D}}{\gamma} \left(c_s \int_{\eta}^{\infty} [f(x) - e^{t\Delta_D}f(x) ]t^{-1-\frac{s}{2}} dt\right) d\gamma
\\&= J_{\epsilon} (\l^s)_{\eta} f(x).
\end{aligned}
\ee We point out that interchanging integrals in the above calculations is allowed by Fubini's Theorem as the integrands do not have any singularities on the domains of integration and all these above integrals are actually bounded by a constant multiple of $\|f\|_{L^1(\Omega)}$ (where the constant depends on $\epsilon$ and $\eta$). Since $(\l^s)_{\eta} J_{\epsilon} f$ converges in $\eta$ to $\l^s J_{\epsilon} f$ strongly in $L^2(\Omega)$ (see \cite{caffarelli2016fractional}),  we infer the existence of a subsequence $(\l^s)_{\eta_k} J_{\epsilon} f$ that converges pointwise to $\l^s J_{\epsilon} f$ for a.e. $x \in \Omega$. Moreover, in view of the linearity of  the operator $J_{\epsilon}$ and its uniform-in-$\epsilon$ boundedness  on $L^2$, we have 
\be 
\|J_{\epsilon} (\l^s)_{\eta_k} f - J_{\epsilon} \l^s f\|_{L^2} \le C\|(\l^s)_{\eta_k} f - \l^s f\|_{L^2}.
\ee But $(\l^s)_{\eta_k} f$ converges strongly in $L^2$ to $\l^s f$, so $J_{\epsilon} (\l^s)_{\eta_k} f$ converges strongly in $L^2$ to $J_{\epsilon} \l^s f$. Consequently, there exists another subsequence $J_{\epsilon} (\l^s)_{{\eta_k}_l} f$ 
that converges to $J_{\epsilon} \l^s f$  a.e. in $\Omega$. By making use of the identity $(\l^s)_{{\eta_k}_l} J_{\epsilon} f = J_{\epsilon} (\l^s)_{{\eta_k}_l}  f$ and taking the limit in $l$, we deduce that $\l^s J_{\epsilon} f(x) = J_{\epsilon} \l^s f(x)$ for a.e. $x \in \Omega$ when $s\in (0,2)$.

Now we consider the case $s=2$. Using the fact $\l^2 = -\Delta_D$, we write $\l^2 J_\epsilon f(x) = -\Delta_D J_\epsilon f(x)$. We prove that $-\Delta_D e^{t\Delta_D} f(x) = e^{t\Delta_D} (-\Delta_D) f(x)$ for all $t>0$ and infer that $\l^2 J_\epsilon f(x) = J_\epsilon \l^2 f(x)$. Indeed, by virtue of \eqref{eqn:heat-kernel},  we have
\begin{align*}
    -\Delta_{D,x} H_D(x,y,t) &= \sum\limits_{j=1}^\infty e^{-t\lambda_j} (-\Delta_D) w_j(x) w_j (y) =  \sum\limits_{j=1}^\infty e^{-t\lambda_j} \lambda_j w_j(x) w_j(y) 
    \\
    &= \sum\limits_{j=1}^\infty e^{-t\lambda_j}  w_j(x) (-\Delta_D) w_j (y) = -\Delta_{D,y} H_D(x,y,t)
\end{align*}
for $t>0$. This identity, together with \eqref{eqn:etdelta}, yield
\begin{align*}
    -\Delta_D e^{t\Delta_D} f(x) &=  \int_\Omega (- \Delta_{D,x})H_D(x,y,t) f(y) dy =  \int_\Omega (- \Delta_{D,y})H_D(x,y,t) f(y) dy 
    \\
    &= \int_\Omega H_D(x,y,t) (-\Delta_{D}) f(y) dy = e^{t\Delta_D} (-\Delta_D) f(x),
\end{align*}
after integrating by parts and making use of the vanishing of the heat kernel on the boundary of $\Omega$ together with the homogeneous Dirichlet boundary conditions obeyed by $f$.

Next, we consider the case $s\in (2,4]$.  Since $\l^s$ and $J_{\epsilon}$ commute for $s \in (0,2]$, we have 
\begin{equation}
    \l^s J_\epsilon f(x) = \l^{s-2} \l^{2} J_\epsilon f(x) =\l^{s-2} J_\epsilon \l^{2} f(x) = J_\epsilon \l^{s-2} \l^{2} f(x) =  J_\epsilon\l^s f(x).
\end{equation} By repeating this procedure,  we obtain that $\l^s J_{\epsilon} f (x) = J_{\epsilon} \l^s f(x)$ for all $s>0$ provided that $f\in \mathcal D(\l^s)$.

\end{proof}

\beg{prop}\label{prop:Lp}
\label{lpb} The operator $J_{\epsilon}$ is of strong type $(p,p)$ for any $p \in [1,\infty]$. In other words, for any $p \in [1,\infty]$, there is a positive constant $C>0$ depending only on $p$
such that 
\be 
\|J_{\epsilon} \theta\|_{L^p} \le C\|\theta\|_{L^p}
\ee for any $\theta \in L^p$.
\end{prop}

\begin{proof}
    The boundedness of $J_{\epsilon}$ on $L^2$ follows from Proposition \ref{regop}. Moreover, $J_{\epsilon}$ is bounded on $L^{\infty}$ and $L^1$, a fact that follows from the maximum principle. Indeed, we have 
    $$
|J_{\epsilon} \theta (x)| 
= \left|\frac1{\ln \epsilon} \int_{\epsilon}^{\frac{1}{\epsilon}} t^{-1} \int_{\Omega} H_D(x,y,t)\theta(y) dy dt 
\right|
\le \|\theta\|_{L^{\infty}} \left|\frac1{\ln \epsilon} \int_{\epsilon}^{\frac{1}{\epsilon}} t^{-1}  dt 
\right| \le 2\|\theta\|_{L^{\infty}}
    $$ for any $\theta \in L^{\infty}$, and 
      $$
\int_{\Omega} |J_{\epsilon} \theta (x)| dx
\le 
\frac{1}{|\ln \epsilon|} \int_{\epsilon}^{\frac{1}{\epsilon}} t^{-1} \int_{\Omega} \left(\int_{\Omega} H_D(x,y,t) dx\right) |\theta(y)| dy dt 
\le \|\theta\|_{L^{1}} \frac1{|\ln \epsilon|} \int_{\epsilon}^{\frac{1}{\epsilon}} t^{-1}  dt 
 =2\|\theta\|_{L^{1}}
    $$ for any $\theta \in L^1$.
    By the Marcinkiewicz interpolation theorem, we infer that $J_{\epsilon}$ is bounded on $L^p$ spaces for any $p \in [1,\infty]$. 
\end{proof}

Propositions~\ref{prop:commute} and \ref{prop:Lp} together imply the following theorem.

\begin{thm}\label{thm:Lp}
   {
   Let $\epsilon \in (0,1)$, $s>0$, $p \in [1,\infty]$. For $f\in\mathcal D(\l^s)$ such that $\l^s f \in L^p$, there is a positive constant $C>0$ depending only on $p$ 
such that 
\begin{equation}
    \|\Lambda_D^s J_\epsilon f \|_{L^p} \leq C \|\Lambda_D^s f \|_{L^p}.
\end{equation}
}
\end{thm}

\section{Fractional Product and Nonlinear Estimates} \la{sec4}

We first prove product estimates in fractional Sobolev spaces. 

\begin{prop}\label{prop:product-1}
    Let $\beta \in (0,1)$ and $p\in[1,\infty)$. For smooth functions $g$ and $h$ defined on $\Omega$, it holds that
    \begin{equation}\label{product:1}
        \|gh\|_{W^{\beta,p}} \leq C \left(\|g\|_{W^{\frac2{pk_1}+\beta,\frac{pk_1}{k_1-1}}} \|h\|_{L^{pk_1}} + \|g\|_{L^{pk_2}}\|h\|_{W^{\frac2{pk_2}+\beta,\frac{pk_2}{k_2-1}}}  \right)
    \end{equation}
    for all $k_1, k_2 \in (\frac{2}{p(1-\beta)},\infty]$.
\end{prop}

\begin{proof}
    Thanks to \eqref{def:frac-sob}, we can write
    \begin{equation*}
        \|gh\|_{W^{\beta,p}}^p = \|gh\|_{L^p}^p + \int_\Omega\int_\Omega \frac{|g(x)h(x)-g(y)h(y)|^p}{|x-y|^{2+\beta p}} dxdy.
    \end{equation*}
    Using the inequality $|a+b|^p \leq C(|a|^p + |b|^p)$ that holds for all $p\geq 1$, we have
    \begin{align*}
        |g(x)h(x)-g(y)h(y)|^p = &|g(x)(h(x)-h(y)) + (g(x)-g(y))h(y)|^p 
        \\
        \leq & C \left(|g(x)|^p |h(x)-h(y)|^p + |g(x)-g(y)|^p |h(y)|^p \right).
    \end{align*}
    We first consider the term $|g(x)-g(y)|^p |h(y)|^p$ in the double integral.
    For $k_1>1$, we use H\"older's inequality with exponents $k_1$ and  $k_1^*=\frac{k_1}{k_1-1}$ and bound 
    \begin{align*}
         &\int_\Omega\int_\Omega \frac{|g(x)-g(y)|^p}{|x-y|^{2+\beta p}} |h(y)|^p dydx 
         \\
         \leq & \left( \int_\Omega\int_\Omega \frac{|g(x)-g(y)|^{pk_1^*}}{|x-y|^{(2+\beta p)k_1^*}} dydx \right)^{\frac1{k_1^*}} \left( \int_\Omega\int_\Omega |h(y)|^{pk_1} dydx\right)^{\frac1{k_1}}
         \\
         = & \left( \int_\Omega\int_\Omega \frac{|g(x)-g(y)|^{pk_1^*}}{|x-y|^{2+pk_1^*(\frac{2}{pk_1}+\beta)}} dydx \right)^{\frac1{k_1^*}} \left( \int_\Omega\int_\Omega |h(y)|^{pk_1} dydx\right)^{\frac1{k_1}}
         \\
         \leq & C \|g\|_{W^{\frac2{pk_1}+\beta,\frac{pk_1}{k_1-1}}}^p \|h\|_{L^{pk_1}}^p,
    \end{align*}
    provided that $\frac2{pk_1}+\beta <1 \Leftrightarrow k_1 > \frac{2}{p(1-\beta)}$.  The second term $|g(x)|^p |h(x)-h(y)|^p$ can be estimated similarly with some $k_2$ satisfying the same condition. The lower order term $\|gh\|_{L^p}$ can be bounded from above by the right-hand side of \eqref{product:1} using the standard H\"older inequality.
\end{proof}

In view of the Sobolev inequality, we deduce the following estimate:

\begin{cor}
    Let $\beta \in (0,1)$ and $p\in[1,\infty)$. For smooth functions $g$ and $h$ defined on $\Omega$, it holds that
    \begin{equation}
        \|gh\|_{W^{\beta,p}} \leq C \left(\|g\|_{H^{1+\beta-\frac2p+\frac4{pk_1}}} \|h\|_{H^{1-\frac2{pk_1}}} + \|g\|_{H^{1-\frac2{pk_2}}}\|h\|_{H^{1+\beta-\frac2p+\frac4{pk_2}}} \right)
    \end{equation}
    for all $k_1, k_2 \in (\frac{2}{p(1-\beta)},\infty)$.
\end{cor}

Using the representation formula \eqref{prodfor1} and Proposition \ref{prop:product-1}, we prove the following product estimates for fractional powers of the Laplacian:

\beg{thm}\label{thm:product-1}
Let $g$ and $h$ be some functions defined on $\Omega$.   
\begin{enumerate}
    \item Let $\beta\in(0,1)$ such that $\beta \neq 1/2$. Let $g\in W^{\frac1{k_1}+\beta, \frac{2k_1}{k_1-1}} \cap L^{2k_2}$ and $h\in  W^{\frac1{k_2}+\beta, \frac{2k_2}{k_2-1}} \cap L^{2k_1}$ for some $k_1, k_2 \in (\frac{1}{1-\beta},\infty]$. Assume, in addition, hat $gh$ vanishes on the boundary when $\beta\in (\frac12,1)$. Then it holds that
    \begin{equation}\label{prod:1}
        \|\l^{\beta} (gh)\|_{L^2} \leq C\left(\|g\|_{W^{\frac1{k_1}+\beta,\frac{2k_1}{k_1-1}}} \|h\|_{L^{2k_1}} + \|g\|_{L^{2k_2}}\|h\|_{W^{\frac1{k_2}+\beta,\frac{2k_2}{k_2-1}}}  \right).
    \end{equation}
    
    \item Let $\beta\in(0,1)$. Let $g \in H^{\beta}(\Omega) \cap L^{\infty}(\Omega)$ and $h \in \mathcal{D}(\l^{\beta}) \cap L^{\infty}(\Omega)$. Then it holds that 
\be \label{fractionalproduct}
\|\l^{\beta} (gh)\|_{L^2} 
\le C\|g\|_{L^{\infty}} \|\l^{\beta} h\|_{L^2} 
+ C\|h\|_{L^{\infty}} \|g\|_{H^{\beta}}.
\ee  

\item Let $\beta\in (0,1)$ and $\gamma \in [0,1]$ such that $\gamma> \beta$. Let $g \in C^{0,\gamma}$ and $h\in \mathcal D(\l^\beta)$. Then it holds that 
\be \label{fractionalproductt}
\|\l^{\beta} (gh)\|_{L^2} 
\le C\|g\|_{L^{\infty}} \|\l^{\beta} h\|_{L^2} 
+ C\|h\|_{L^{2}} [g]_{C^{0,\gamma}} .
\ee  

\end{enumerate}
\end{thm}

\begin{proof} The proof of (1) follows from the identifications $\mathcal D(\l^\beta)=H^\beta$ when $\beta\in(0,\frac12)$ and $\mathcal D(\l^\beta)=H_0^\beta$ when $\beta\in(\frac12,1)$, and Proposition~\ref{prop:product-1}.

For (2), in view of the representation formula \eqref{prodfor1}, we have 
    \be 
\|\l^{\beta} (gh)\|_{L^2}^2
= \int_{\Omega} \int_{\Omega} \left(g(x)h(x) - g(y)h(y)\right)^2 K_{\beta} (x,y) dxdy
+ \int_{\Omega} g(x)^2 h(x)^2 B_\beta(x) dx,
    \ee which amounts to 
    \be 
    \beg{aligned}
\|\l^{\beta} (gh)\|_{L^2}^2
&= \int_{\Omega} \int_{\Omega} \left(g(x)h(x) - g(y) h(x) + g(y) h(x) - g(y)h(y)\right)^2 K_{\beta} (x,y) dxdy
\\&\quad\quad+ \int_{\Omega} g(x)^2 h(x)^2 B_\beta(x) dx
\end{aligned}
    \ee after adding and subtracting $g(y)h(x)$ in the first double integral. By making use of the algebraic inequality $(a+b)^2 \le 2a^2 + 2b^2$ that holds for all $a, b \in \R$, we estimate
    \be 
    \beg{aligned}
\|\l^{\beta} (gh)\|_{L^2}^2
&\le 2 \int_{\Omega} \int_{\Omega} (g(x) - g(y))^2 h(x)^2 K_{\beta}(x,y) dxdy
\\&\quad\quad+ 2\int_{\Omega} \int_{\Omega} (h(x) - h(y))^2 g(y)^2 K_{\beta}(x,y) dxdy 
\\&\quad\quad\quad\quad+ \int_{\Omega} g(x)^2 h(x)^2 B_\beta(x) dx.
\end{aligned}
    \ee In view of the kernel estimate \eqref{prodfor2}, we obtain 
    \be 
    \beg{aligned}
\|\l^{\beta}(gh)\|_{L^2}^2 
&\le C\|h\|_{L^{\infty}}^2 \int_{\Omega}\int_{\Omega} \frac{(g(x) - g(y))^2}{|x-y|^{2+2\beta}} dxdy \\&\quad\quad+ 2\|g\|_{L^{\infty}}^2 \left\{\int_{\Omega}\int_{\Omega} (h(x) - h(y))^2 K_{\beta}(x,y)dxdy + \int_{\Omega} h(x)^2B_\beta(x) dx \right\},
    \end{aligned}
    \ee yielding the desired bound \eqref{fractionalproduct}. 
    
 As for (3), we make use of the commutator estimate 
\be 
\|\l^{\beta} (gh) - g \l^{\beta} h\|_{L^2} \le C[g]_{C^{0,\gamma}} \|h\|_{L^2}
\ee (see Theorem 2.6 in \cite{nguyen2018global}) and estimate the product
\be 
\|\l^{\beta} (gh)\|_{L^2}
\le \|\l^{\beta} (gh) - g \l^{\beta} h\|_{L^2} + \|g \l^{\beta} h\|_{L^2}
\le C[g]_{C^{0,\gamma}} \|h\|_{L^2} + C\|g\|_{L^{\infty}} \|\l^{\beta} h\|_{L^2}.
\ee 

\end{proof}

Next, we prove the following fractional trilinear estimate:

\beg{thm} \la{thm:nonlinear} Let $\alpha \in (1,2)$. Let $\epsilon_0$ be a sufficiently small quantity such that $\alpha > 1+ 2\epsilon_0$.  Suppose $v$ and $w$ are smooth functions that vanish on the boundary $\partial \Omega$, and $u$ is a smooth vector field.
Then it holds that 
\be 
|(\l^{1-\frac{\alpha}{2}} (u \cdot \na v), \l^{1 + \frac{\alpha}{2}}w)_{L^2(\Omega)}|
\le C\|\l^{1+\frac{\alpha}{2}} w\|_{L^2(\Omega)} \|u\|_{H^1(\Omega)} \|\l^{2 - \frac{\alpha}{2} + \epsilon_0} v\|_{L^2(\Omega)}.
\ee 
\end{thm}

Theorem~\ref{thm:nonlinear} follows from the following lemma.

\beg{lem}\label{lem-1}
Let $\alpha \in (1,2)$. Let $\epsilon_0$ be a sufficiently small quantity such that $\alpha > 1+ 2\epsilon_0$.  Suppose $v \in \mathcal{D}(\l^{2-\frac{\alpha}{2} + \epsilon_0})$, $w \in \mathcal{D}(\l^{1 + \frac{\alpha}{2}})$, and $u \in H^1(\Omega)$.
For any $\epsilon>0$, it holds that 
\be \la{innerproductb}
|(u \cdot \na v, -\Delta J_{\epsilon} w)_{L^2(\Omega)}|
\le C\|\l^{1+\frac{\alpha}{2}} w\|_{L^2(\Omega)} \|u\|_{H^1(\Omega)} \|\l^{2 - \frac{\alpha}{2} + \epsilon_0} v\|_{L^2(\Omega)}.
\ee 
\end{lem}

\begin{proof} We drop the regularizer $J_{\epsilon}$ throughout the proof for simplicity. First note that since $2-\frac\alpha2+\epsilon_0>1$ and  $v \in \mathcal{D}(\l^{2-\frac{\alpha}{2} + \epsilon_0})$, it follows that $v$ has a vanishing trace on $\pa \Omega$.
 Let $v'(x)$ be the vector field defined by 
\be 
v'(x) =\begin{cases}
\nabla v(x), \hspace{3cm} x \in \Omega,
\\0 \hspace{3.5cm} x \in \RR^2 \setminus \Omega.
\end{cases}
\ee Since $\alpha/2 < 1$ and $\na w \in H^{\frac{\alpha}{2}}(\Omega)$, the result of \cite{zhou2015fractional} implies the existence of an extension $w'$ of $\na w$ obeying $w'(x) = \na w(x)$ in $\Omega$ and $\|w'\|_{H^{\fr{\alpha}{2}}(\RR^2)} \le C\|\na w\|_{H^{\fr{\alpha}{2}}(\Omega)}$ where $C$ is a universal positive constant.
Also, as $u \in H^1(\Omega)$, there is an extension $\tilde{u}$ of $u$ such that $\tilde{u}(x) = u (x)$ in $\Omega$ and $\|\tilde{u}\|_{H^1(\RR^2)} \le C\|{u}\|_{H^1(\Omega)}$.
Consequently, we have 
\be 
\beg{aligned}
(u \cdot \na v, -\Delta w)_{L^2(\Omega)}
= (\tilde{u} \cdot v', - \na \cdot  w')_{L^2(\RR^2)} = (\na (\tilde u \cdot v'), w')_{L^2(\RR^2)}.
\end{aligned}
\ee Denoting by $\Lambda_{\RR^2}$ the fractional Laplacian of order 1 on the whole space $\RR^2$ (which is defined as a Fourier multiplier) and using the fact that $\Lambda_{\RR^2}^{-1}$ and $\na$ commute on $\RR^2$, we can rewrite this latter inner product as 
\be 
(u \cdot \na v, -\Delta w)_{L^2(\Omega)}
= (\Lambda_{\RR^2} \na \Lambda_{\RR^2}^{-1} (\tilde{u} \cdot v'), w')_{L^2(\RR^2)} 
= (\Lambda_{\RR^2}^{1-\fr{\alpha}{2}} \na \Lambda_{\RR^2}^{-1} (\tilde{u} \cdot v'), \Lambda_{\RR^2}^{\fr{\alpha}{2}} w')_{L^2(\RR^2)}.
\ee A straightforward application of the Cauchy-Schwarz inequality gives rise to 
\be 
|(u \cdot \na v, -\Delta w)_{L^2(\Omega)}|
\le \|\Lambda_{\RR^2}^{1-\fr{\alpha}{2}} \na \Lambda_{\RR^2}^{-1} (\tilde{u} \cdot v')\|_{L^2(\RR^2)} \| \Lambda_{\RR^2}^{\fr{\alpha}{2}} w'\|_{L^2(\RR^2)}.
\ee 
We estimate
\be 
\|\Lambda_{\RR^2}^{\fr{\alpha}{2}} w'\|_{L^2(\RR^2)} \le C\|w'\|_{H^{\fr{\alpha}{2}}(\RR^2)}
\le C\|\na w\|_{H^{\fr{\alpha}{2}}(\Omega)}
\le C\|w\|_{H^{1+\frac{\alpha}{2}}(\Omega)} \le C\|\l^{1+\frac{\alpha}{2}} w\|_{L^2(\Omega)}
\ee where the last inequality follows from the fact that $w \in \mathcal{D}(\l^{1+\frac{\alpha}{2}})$ and the continuous embedding of $\mathcal{D}(\l^{1+\frac{\alpha}{2}})$ into $H^{1+\frac{\alpha}{2}}(\Omega)$.
Using the fact that $\na \Lambda_{\RR^2}^{-1}$ and $\Lambda_{\RR^2}^{1-\frac{\alpha}{2}}$ commute and the boundedness of the Riesz transform $\na \Lambda_{\RR^2}^{-1}$ on $L^2(\RR^2)$, we have 
\be 
\|\Lambda_{\RR^2}^{1-\fr{\alpha}{2}} \na \Lambda_{\RR^2}^{-1} (\tilde{u} 
 \cdot v')\|_{L^2(\RR^2)}
\le C\| \Lambda_{\RR^2}^{1-\fr{\alpha}{2}} (\tilde{u} \cdot v')\|_{L^2(\RR^2)}.
\ee By making use of the fractional product estimates in $\mathbb R^2$, we bound
\be 
\beg{aligned}
\|\Lambda_{\RR^2}^{1-\fr{\alpha}{2}} \na \Lambda_{\RR^2}^{-1} (\tilde{u} 
 \cdot v')\|_{L^2(\RR^2)}
 &\le C\|\tilde{u} \|_{W^{1-\frac{\alpha}{2}, \frac{4}{2-\alpha}}(\RR^2)} \|v' \|_{L^{\frac{4}{\alpha}}(\RR^2)} + C\|\tilde{u}\|_{L^{r}(\RR^2)} \|v'\|_{W^{1-\frac{\alpha}{2},2+\kappa}(\RR^2)}
 \end{aligned}
\ee where $\kappa$ is chosen so that is $H^{1-\frac{\alpha}{2} + \epsilon_0}$ continuously embedded in $W^{1-\frac{\alpha}{2}, 2+ \kappa}$ and $r$ is the H\"older exponent obeying $\frac{1}{r} + \frac{1}{2+\kappa} = \frac{1}{2}$. 
Using the continuous Sobolev embedding of $H^1(\RR^2)$ into $W^{1-\frac{\alpha}{2}, \frac{4}{2-\alpha}}$, we have 
\be 
\|\tilde{u} \|_{W^{1-\frac{\alpha}{2}, \frac{4}{2-\alpha}}(\RR^2)}
\le C\|\tilde{u}\|_{H^1(\RR^2)}
\le C\|u\|_{H^1(\Omega)}.
\ee As $v'(x)=\nabla v(x)$ in $\Omega$, and since $H^{2-\frac{\alpha}{2}}(\Omega)$ is continuously embedded in $W^{1, \frac{4}{\alpha}}(\Omega)$, it holds that 
\be 
\|v'\|_{L^{\frac{4}{\alpha}}(\RR^2)}
= \|\na v\|_{L^{\frac{4}{\alpha}}(\Omega)}
\le C\|v\|_{H^{2-\frac{\alpha}{2}}(\Omega)}
\le C\|\l^{2-\frac{\alpha}{2}} v\|_{L^2(\Omega)}.
\ee In view of the continuous embedding of $H^1(\RR^2)$ into $L^r(\RR^2)$, we estimate 
\be 
\|\tilde{u}\|_{L^r(\RR^2)} \le C\|\tilde{u}\|_{H^1(\RR^2)} \le C\|u\|_{H^1(\Omega)}.
\ee As $0 \le 1- \frac{\alpha}{2} + \epsilon_0 < \frac{1}{2}$, we infer that $\na v \in \mathcal{D}(\l^{1- \frac{\alpha}{2} + \epsilon_0})$ and we apply Proposition 2.3 in \cite{stokols2020holder} to bound
\be 
\|v'\|_{W^{1-\frac{\alpha}{2}, 2+\kappa}(\RR^2)} \le C\|v'\|_{H^{1-\frac{\alpha}{2} + \epsilon_0}(\RR^2)} 
\le C\|\l^{1-\frac{\alpha}{2} + \epsilon_0} \na {v}\|_{L^2(\Omega)}
\le C\|\l^{2-\frac{\alpha}{2} + \epsilon_0} {v}\|_{L^2(\Omega)}.
\ee Putting all these estimates together, we obtain \eqref{innerproductb}. 
\end{proof}

\section{A Poincar\'e Inequality for the Fractional Laplacian} \la{sec5}

We recall the pointwise C\'ordoba-C\'ordoba inequality on bounded smooth domains: 

\beg{prop}[\cite{constantin2016critical}] \la{corcor} Let $0 \le s < 2$. There exists a constant $c > 0$ depending only on the domain  $\Omega$ and on $s$, such that for any $C^1$ convex function $\Phi$ satisfying $\Phi(0) = 0$, and any function $f \in \mathcal{D}(\l^s)$ with $\Phi(f) \in \mathcal{D}(\l^s)$, the inequality 
\be 
\Phi'(f) \l^s f - \l^s (\Phi(f)) \ge \fr{c}{d(x)^s} (f \Phi'(f) - \Phi(f))
\ee holds pointwise in $\Omega$.
\end{prop}

\begin{rem}
    In the original statement \cite{constantin2016critical} it is required that $\Phi \in C^2$. Indeed, following the proof in \cite{constantin2017remarks}, one only needs $\Phi$ to be $C^1$ and convex.
\end{rem}

We use this pointwise inequality to prove a Poincar\'e inequality for the fractional Laplacian in $L^p(\Omega)$. We need first the following two elementary lemmas:

\beg{prop} \la{prop:nonnegativity}
Let $s \in (0,2)$. Let $f$ be a nonnegative function such that $f \in \mathcal{D}(\l^s)$. Then it holds that 
\be 
\int_{\Omega} \l^s f(x) dx \ge 0.
\ee 
\end{prop}

\beg{proof} For $\eta \in (0,1)$, recall the truncated fractional Laplacian defined in \eqref{eqn:trun-Lap}:
\be 
(\l^s)_{\eta} f(x) = c_s \int_{\eta}^{\infty}  [f(x) - e^{t\Delta_D} f(x)] t^{-1-\frac{s}{2}} dt.
\ee We have
\be 
\int_{\Omega} (\l^s)_{\eta} f(x) dx 
= c_s \int_{\Omega} \int_{\eta}^{\infty}\left[ f(x) - \int_{\Omega} H_D(x,y,t) f(y) dy \right] t^{-1-\frac{s}{2}} dt dx.
\ee Now we interchange the order of integration, which is allowed by Fubini's theorem as the integrand does not have any singularities. We obtain 
\be 
\int_{\Omega} (\l^s)_{\eta} f(x) dx 
= c_s \int_{\eta}^{\infty} \left[\int_{\Omega} f(x) dx - \int_{\Omega} \left(\int_{\Omega} H_D(x,y,t) dx \right) f(y) dy \right] t^{-1- \frac{s}{2}}dt.
\ee By the symmetry of the heat kernel and the maximum principle, it holds that 
\be 
0 \le \int_{\Omega} H_D(x,y,t) dx = \int_{\Omega} H_D(y,x,t) dx \le 1
\ee for all $y \in \Omega$ and $t \ge 0$. Due to the nonnegativity of $f$, it follows that 
\be 
0 \le \int_{\Omega} \left(\int_{\Omega} H_D(x,y,t) dx \right) f(y) dy \le \int_{\Omega} f(y) dy, 
\ee and consequently,
\be 
\int_{\Omega} (\l^s)_{\eta} f(x) dx 
\ge c_s \int_{\eta}^{\infty} \left[\int_{\Omega} f(x) dx - \int_{\Omega} f(y) dy \right] t^{-1-\frac{s}{2}} dt \ge 0.
\ee Finally, $(\l^s)_{\eta} f$ converges strongly in $L^2$ to $\l^s f$, so 
\be 
\left|\int_{\Omega} (\l^s)_{\eta} f(x) dx - \int_{\Omega} \l^s f(x) dx \right|
\le |\Omega|^{\frac{1}{2}} \|(\l^s)_{\eta} f - \l^s f\|_{L^2} \rightarrow 0.
\ee Therefore, we conclude that 
\be 
\int_{\Omega} \l^s f(x) dx \ge 0. 
\ee 
\end{proof}

\begin{prop} \label{prop:regularity}
    Let $s\in (0,2]$ and $\beta>0$. Consider a function $q\in \mathcal D(\l^s)$. Then $|q|^\beta\in \mathcal D(\l^s)$ holds if (i)  $q$ is Lipschitz  continuous, and $0<\beta<1$ and $0<s<\beta$; (ii) $\beta=1$ and $s=1$; (iii) $q \in L^\infty(\Omega)$, and $1<\beta<2$ and $s=1$; (iv) $\beta\geq 2$ and $s= 2$.
\end{prop}

\begin{proof}
    We start with case (i). When $0<\beta\leq \frac12$, it follows that $s\in(0,\frac12)$, thus the domain $\mathcal D(\l^s)$ is identified with $H^s$, and we do not need to worry about the vanishing on the boundary. For $\beta\in(\frac12,1)$ and $s\in(\frac12,\beta)$,  the domain $\mathcal D(\l^s)$ is identified with $H_0^s$. In this situation, $q\in \mathcal D(\l^s)$, so $q$ vanishes on the boundary, and so does $|q|^{\beta}$.  Therefore, we only need to check whether $|q|^\beta \in H^s$. 

    By making use of the $\beta$-H\"older continuity of the function $f(x) = |x|^{\beta}$, the reverse triangle inequality, and the Lipschitz continuity of $q$, we have
    \begin{equation}
        \Big| |q(x)|^\beta - |q(y)|^\beta \Big| \leq C \Big| |q(x)| - |q(y)| \Big|^\beta \leq C\Big| q(x) - q(y)\Big|^\beta \leq C|x-y|^\beta. 
    \end{equation}
    Thus, we obtain
    \begin{equation}
        \frac{\Big| |q(x)|^\beta - |q(y)|^\beta \Big|^2}{|x-y|^{2+2s}} \leq C |x-y|^{-2-2s+2\beta}.
    \end{equation}
   Since $s<\beta$, it follows that $|x-y|^{-2-2s+2\beta} \in L^1(\Omega\times\Omega)$, and therefore $|q|^\beta \in W^{s,2}=H^s$.

    Case (ii) follows directly from the fact that $|q|\in \mathcal D(\l) = H_0^1$ when $q\in D(\l) = H_0^1$.

    For case (iii), note that $|\nabla(|q|^\beta)|^2 = \beta^2 |q|^{2\beta-2} |\nabla q|^2$, which is integrable on $\Omega$ since $q \in D(\l) \cap L^\infty$.

    For case (iv), one can compute
    \[
     \Delta (|q|^\beta) = (\beta-1)|q|^{\beta-2}|\nabla q|^2 + |q|^{\beta-1} \frac{q}{|q|} \Delta q.
    \]
    Since $q \in \mathcal{D}(\l^2) = H^2 \cap H_0^1$, it follows that $q \in L^{\infty}$ in view of the continuous Sobolev embedding of $H^2$ in $L^{\infty}$. As $\beta \ge 2$, we infer that $\Delta (|q|^{\beta}) \in L^2$. Since $|q|^{\beta}$ vanishes on the boundary, we deduce furthermore that $|q|^\beta \in \mathcal D(\l^2) = H^2 \cap H_0^1$.
\end{proof}

\begin{thm} \la{thm:poin}
Let $p \geq 2$, $0 < s < 2$, and $q$ be a function smooth up to the boundary such that $q\in \mathcal D(\l^{s})$.  
Then 
\be 
\int_{\Omega} q(x)|q(x)|^{p-2} \l^{s} q(x) dx 
\ge c_1 \|\l^{\fr{s}{2}}(|q|^{p/2})\|_{L^2}^2 + c_2\|q\|_{L^p}^p
\label{Poinc} 
\ee holds, where 
\begin{equation*}
    (c_1, c_2)=
    \begin{cases}
      (\frac1p, \frac1p {\lambda_1^{\frac{s}2}}), & \text{if}\ p=2, 4, \\
      (0, C_{\Omega,s} (1-\frac2p)), & \text{if $2 < p < 4$ and $s > 1$, or $2<p<3$ and $p-2\leq s \leq 1$,}\\ 
      (2-\frac4p, C_{\Omega,s} (\frac4p-1)), &\text{if $2<p<3$ and $0<s<p-2$, or $3 \leq p < 4$ and $s \le 1$,}\\ 
      (\frac4p, C_{\Omega,s}(1-\frac4p)), &\text{if $p > 4$.}
    \end{cases}
  \end{equation*}

\end{thm}

\begin{proof} The cases $p=2$ and $p=4$ are trivial. Indeed,  
\be 
\int_{\Omega} q(x) \l^s q(x) dx = \|\l^{\fr{s}{2}} q\|_{L^2}^2 \geq \fr{1}{2} \|\l^{\fr{s}{2}} q\|_{L^2}^2 + \fr{1}{2} \lambda_1^{\frac{s}2}\|q\|_{L^2}^2,
\ee 
in view of the continuous embedding $\mathcal{D}(\l^{\fr{s}{2}}) \subset L^2$, whereas 
\be 
\int_{\Omega} |q(x)|^2 q(x) \l^{s}q(x) dx
\ge \fr{1}{2} \int_{\Omega} |q(x)|^2  \l^{s} q(x)^2 dx \geq \fr{1}{4} \|\l^{\fr{s}{2}} q^2\|_{L^2}^2 + \fr{1}{4} \lambda_1^{\frac{s}2}\|q\|_{L^4}^4
\ee
in view of the C\'ordoba-C\'ordoba inequality.

Now suppose that $p > 2$ with $p \neq 4$. 
We note that 
\be 
\int_{\Omega} |q|^{p-2} (q \l^{s}{q}) dx 
\geq \fr{1}{2} \int_{\Omega} |q|^{p-2} \l^{s} {q^{2}} dx
\ee 
and distinguish three different cases:

\textbf{Case 1.} $2 < p < 4$ and $s > 1$, or $2<p<3$ and $p-2\leq s \leq 1$. By the C\'ordoba-C\'ordoba inequality and Proposition~\ref{prop:nonnegativity}, we have
\be 
\beg{aligned}
\int_{\Omega} |q|^{p-2} \l^s q^2 dx
&= \int_{\Omega} \left(q^2 \right)^{\frac{p-2}{2}} \l^s q^2 dx
\\&\ge \frac{2}{p} \int_{\Omega} \l^s (|q|^p) dx + \int_{\Omega} \frac{C}{d(x)^s} \left(q^2 (q^2)^{\frac{p-2}2} - \frac{2}{p} (q^2)^{\frac{p}{2}}\right) dx
\\&\ge \frac{C}{diam(\Omega)^s} \left(1 - \frac{2}{p} \right) \|q\|_{L^p}^p.
\end{aligned}
\ee

\textbf{Case 2.} $2<p<3$ and $0<s<p-2$, or $3 \leq p < 4$ and $s \le 1$. In this case, thanks to Proposition~\ref{prop:regularity}, one has $|q|^{p-2} \in \mathcal D(\l^s)$. Then integrating by parts and using the  C\'ordoba-C\'ordoba inequality, we have 
\beg{align}
&\int_{\Omega} |q|^{p-2} \l^{s} {q^{2}} dx 
= \int_{\Omega} |q|^2 \l^{s}(|q|^{p-2} )  dx  
=  \int_{\Omega} |q|^{\fr{p}{2}} (|q|^{p-2})^{{\fr{4-p}{2(p-2)}}}\l^{s}(|q|^{p-2} )  dx  \nonumber
\\&\ge \int_{\Omega} |q|^{\fr{p}{2}}\left(\fr{2p-4}{p} \l^{s} |q|^{\fr{p}{2}} + \fr{C}{\mathrm{diam}(\Omega)^s}.\fr{4-p}{p} |q|^{\fr{p}{2}} \right)  dx  \nonumber
\\&= \fr{2p-4}{p} \|\l^{\fr{s}{2}}(|q|^{p/2})\|_{L^2}^2 + \fr{C}{\mathrm{diam}(\Omega)^s}\fr{4-p}{p} \|q\|_{L^p}^p
\end{align}  where $C$ is a positive constant depending only on $\Omega$ and $s$. {The key difference between Case 1 and Case 2 is whether $|q|^{p-2} \in \mathcal D(\l^s)$ or not.}

\textbf{Case 3.} $p > 4$. This case is treated exactly as Case 2 but without integrating by parts. In fact, 
\beg{align}
&\int_{\Omega} |q|^{p-2} \l^{s} {q^{2}} dx 
=  \int_{\Omega} |q|^{\fr{p}{2}} (|q|^2)^{\fr{p-4}{4}}\l^{s}(|q|^2 )  dx  \nonumber
\\&\ge \int_{\Omega} |q|^{\fr{p}{2}}\left(\fr{4}{p} \l^{s} |q|^{\fr{p}{2}} + \fr{C}{\mathrm{diam}(\Omega)^s}.\fr{p-4}{p} |q|^{\fr{p}{2}} \right)  dx  \nonumber
\\&= \fr{4}{p} \|\l^{\fr{s}{2}}(|q|^{p/2})\|_{L^2}^2   + \fr{C}{\mathrm{diam}(\Omega)^s}\fr{p-4}{p} \|q\|_{L^p}^p.
\end{align}
\end{proof}

\section{Application: Subcritical SQG Equation}
\la{sec6}

Let $\Omega \subset \RR^2$ be a bounded domain with a smooth boundary, and $-\Delta_D$ be the two-dimensional Laplacian with homogeneous Dirichlet boundary conditions. For $\alpha \in (1,2)$, we consider the forced subcritical surface quasi-geostrophic (SQG) model
\be \label{SQGmodel}
\begin{cases}
\pa_t q + u \cdot \na q + {\l^{\alpha}} q = f
\\ u = R^{\perp} q
\\ q|_{\pa \Omega} = 0
\\q(x,0) = q_0(x)
\end{cases}
\ee on $\Omega$, where $q= q(x,t)$ is a scalar function, $f = f(x)$ is a time-independent bulk forcing satisfying $f|_{\pa \Omega}=0$, $\l^{\alpha} = (-\D_D)^{\fr{\alpha}{2}}$ is the fractional Laplacian of order $\alpha$, and $R^{\perp} = \na^{\perp} \l^{-1} = (-\p_2 \l^{-1},\pa_1 \l^{-1})$ is a rotation of the two-dimensional Riesz transform. Note that when $q$ is regular enough (e.g., $q\in H_0^1(\Omega)$), $(u\cdot n)|_{\pa \Omega}=0$ as the stream function $\psi=\l^{-1}q$ vanishes at the boundary and its gradient is normal to the boundary.

The SQG equation was initially proposed in \cite{constantin1994singular} and its global regularity was addressed in the absence (\cite{caffarelli2010drift, constantin2012nonlinear, kiselev2007global, zelati2016global} and reference therein) and presence (\cite{ constantin2023global, constantin2018local, nguyen2018global} and references therein) of physical boundaries. The long-time behavior of solutions to the unforced SQG equation was studied in \cite{constantin1999behavior} on the whole space for any $\alpha \in (0,2)$ and in \cite{constantin2015long} for the forced critical equation on two-dimensional periodic boxes equipped with periodic boundary conditions. 

In this section, we study the long-time dynamics of the subcritical SQG model in the presence of spatial boundaries.

\subsection{Construction of Weak Solutions in Low Regular Lebesgue Spaces}

In this subsection, we prove the existence of global-in-time weak solutions under weak regularity assumptions imposed on the initial data. To this end, we consider a spectral-parabolic regularization of  system \eqref{SQGmodel} and use it to construct solutions in some Lebesgue spaces. Namely, for $\epsilon \in (0,1)$, we consider the $\epsilon$-approximate system 
\be \label{epsreg}
\begin{cases}
\pa_t q^{\epsilon} + u^{\epsilon} \cdot \na q^{\epsilon} + {\l^{\alpha}} q^{\epsilon} -\epsilon \Delta_D q^{\epsilon} = J_{\epsilon} f
\\ u^{\epsilon} = R^{\perp} q^{\epsilon}
\\ q^{\epsilon}|_{\pa \Omega} = 0
\\q^{\epsilon}(x,0) = J_{\epsilon} q_0(x)
\end{cases}
\ee on $\Omega$.
Using Proposition \ref{regop}, we construct global weak solutions to system \eqref{SQGmodel} with low regularity.

\begin{thm}\label{thm:sqg-1}
Let $q_0 \in \mathcal{D}(\l^{-\fr{1}{2}})$, $f \in \mathcal{D}(\l^{\fr{-1-\alpha}{2}}),$ and $T>0$ be an arbitrary positive time. The initial boundary value problem \eqref{SQGmodel} has a weak solution $q$ on $[0,T]$ such that 
\be \label{th1reg}
q \in L^{\infty} (0,T; \mathcal{D}(\l^{-\fr{1}{2}})) \cap L^2(0,T; \mathcal{D}(\l^{\fr{\alpha -1}{2}})).
\ee 
\end{thm}

\begin{proof}
For each $\epsilon > 0$, the $\epsilon$-regularized system \eqref{epsreg} has a unique global smooth solution satisfying $q^\epsilon\in \mathcal D(\l^s)$ for all $s\geq 0$ and $(-\Delta_D)^{\ell} q^{\epsilon}|_{\pa \Omega} = 0$ for all $\ell \in \NN$ (see \cite{elie}). 
Below we provide {\it a priori} bounds and pass to the limit as $\epsilon \rightarrow 0$. 

We multiply the equation obeyed by $q^{\epsilon}$  in \eqref{epsreg} by $\l^{-1}q^{\epsilon}$, integrate spatially over $\Omega$, and obtain the energy evolution 
\be 
\fr{1}{2} \fr{d}{dt} \|\l^{-\fr{1}{2}}q^{\epsilon}\|_{L^2}^2 + \|\l^{\fr{\alpha - 1}{2}} q^{\epsilon}\|_{L^2}^2 + \epsilon \|\l^{\fr{1}{2}} q^{\epsilon}\|_{L^2}^2 
=- \int_{\Omega} (R^{\perp} q^{\epsilon} \cdot \na q^{\epsilon}) \l^{-1} q^{\epsilon} dx + \int_{\Omega} J_{\epsilon} f \l^{-1} q^{\epsilon} dx  .
\ee Integrating by parts, exploiting the homogeneous Dirichlet boundary condition for $q^{\epsilon}$, using the divergence-free condition obeyed by $u^{\epsilon} = R^{\perp} q^{\epsilon}$, and applying the pointwise cancellation law $R^{\perp} q^{\epsilon} \cdot Rq^{\epsilon} = 0$, we infer that 
\be 
- \int_{\Omega} (R^{\perp} q^{\epsilon} \cdot \na q^{\epsilon}) \l^{-1} q^{\epsilon} dx = \int_{\Omega} (R^{\perp} q^{\epsilon} \cdot \nabla \l^{-1} q^{\epsilon}) q^{\epsilon} dx
= \int_{\Omega} (R^{\perp} q^{\epsilon} \cdot Rq^{\epsilon}) q^{\epsilon} dx = 0
\ee holds. In view of the Cauchy-Schwarz inequality, the uniform-in-$\epsilon$ estimate \eqref{prop32},  and Young inequality for products, we have  
\be 
\left|\int_{\Omega} J_{\epsilon}f \l^{-1} q^{\epsilon} dx \right| 
\le \fr{1}{2} \| \l^{\fr{\alpha - 1}{2}}q^{\epsilon}\|_{L^2}^2 + C \|\l^{\frac{-1-\alpha}{2}} f\|_{L^2}^2.
\ee This yields the energy inequality
\be 
 \fr{d}{dt} \|\l^{-\fr{1}{2}}q^{\epsilon}\|_{L^2}^2 + \|\l^{\fr{\alpha - 1}{2}} q^{\epsilon}\|_{L^2}^2 \le C\|\l^{\frac{-1-\alpha}{2}} f\|_{L^2}^2,
\ee from which we deduce that 
\be 
\beg{aligned}
&\|\l^{-\fr{1}{2}}q^{\epsilon}(t)\|_{L^2}^2 + \int_{0}^{t} \|\l^{\fr{\alpha - 1}{2}} q^{\epsilon}(s) \|_{L^2}^2 ds \\&\quad\quad\le \|\l^{-\fr{1}{2}}J_{\epsilon} q_0\|_{L^2}^2 +C\|\l^{\frac{-1-\alpha}{2}} f\|_{L^2}^2t \le C\|
\l^{-\fr{1}{2}} q_0\|_{L^2}^2 + C\|\l^{\frac{-1-\alpha}{2}} f\|_{L^2}^2t
\end{aligned}
\ee after integrating in time from $0$ to $t$. Consequently, the family of regularized solutions $\left\{q^{\epsilon}\right\}_{\epsilon \in (0,1)}$ is uniformly bounded (in $\epsilon$) in the Lebesgue spaces $L^{\infty}(0,T;\mathcal{D}(\l^{-\fr{1}{2}}))$ and $L^2(0,T; \mathcal{D}(\l^{\fr{\alpha-1}{2}}))$. 

The family of nonlinear terms $u^{\epsilon} \cdot \na q^{\epsilon}$ is uniformly bounded in $L^1(0, T; \mathcal{D}(\l^{-2}))$. Indeed, we have 
\be 
\beg{aligned}
\left|\int_{\Omega} (u^{\epsilon} \cdot \na q^{\epsilon}) \Phi dx \right|
= \left|\int_{\Omega} (u^{\epsilon} q^{\epsilon}) \cdot \na \Phi dx \right|
&\le C\|R^{\perp} q^{\epsilon}\|_{L^{\fr{4}{3-\alpha}}} \|q^{\epsilon}\|_{L^{\fr{4}{3 - \alpha}}} \|\na \Phi\|_{L^{\fr{2}{\alpha - 1}}}
\\&\le C\|\l^{\fr{\alpha - 1}{2}} q^{\epsilon}\|_{L^2}^2 \|\Phi\|_{H^2}
\end{aligned}
\ee for all $\Phi \in \mathcal{D}(\l^2)$, where the first equality follows from the divergence-free condition obeyed by $u^{\epsilon}$ and the last inequality holds due to the boundedness of the Riesz transform on $L^p$ spaces for $p \in (1,\infty)$ and classical continuous Sobolev embeddings. Since the other terms $\l^{\alpha} q^{\epsilon}$,  $\Delta_D q^{\epsilon}$, and $\mathcal J_\epsilon f$ are also uniformly bounded in $L^1(0,T; \mathcal{D}(\l^{-2}))$, so is the family of time derivatives $\left\{\pa_tq^{\epsilon} \right\}_{\epsilon \in (0,1)}$. Due to the compact embedding of $\mathcal{D}(\l^{\fr{\alpha - 1}{2}})$ in $L^2$ for $\alpha \in (1,2)$, and the continuous embedding of $L^2$ in $\mathcal{D}(\l^{-2})$, we can apply the Aubin-Lions theorem and Banach Alaoglu theorem to deduce that $\left\{q^{\epsilon}\right\}_{\epsilon \in (0,1)}$ has a subsequence that converges strongly in $L^2(0,T; L^2)$ and weakly in $L^2(0,T; \mathcal{D}(\l^{\fr{\alpha-1}{2}}))$ to a weak solution $q$ of \eqref{SQGmodel} obeying the desired regularity property \eqref{th1reg}. We omit further details. 
\end{proof}

\subsection{Construction of Unique Strong Solutions}

In this subsection, we prove the uniqueness of solutions to the model \eqref{SQGmodel} provided that the initial data is $L^{p}$ regular for a sufficiently large number $p$.

\beg{thm}\label{thm:sqg-2}
Let $\alpha \in (1,2)$ and $\delta \in (0, \frac{\alpha-1}2)$.  Let $q_0 \in L^{\fr{1}{\delta}}$, $f \in L^{\fr{1}{\delta}}$, and $T>0$ be an arbitrary positive time. The initial boundary value problem \eqref{SQGmodel} has a unique solution $q$ on $[0,T]$ such that 
\be \label{reg:q}
q \in L^{\infty} (0,T; L^{\frac{1}{\delta}}) \cap L^2(0,T; \mathcal{D}(\l^{\frac{\alpha}{2}})).
\ee 
\end{thm}

\begin{proof}
    We multiply the $q^{\epsilon}$ equation in \eqref{epsreg} by $q^{\epsilon}$ and integrate spatially over $\Omega$. We obtain the differential inequality
    \be \la{strong1}
\frac{d}{dt}\|q^{\epsilon}\|_{L^2}^2
+ \|\l^{\fr{\alpha}{2}} q^{\epsilon}\|_{L^2}^2 \le C\|\l^{-\fr{\alpha}{2}} f\|_{L^2}^2
    \ee after making use of the cancellation law 
    \be 
\int_{\Omega} u^{\epsilon} \cdot \na q^{\epsilon} q^{\epsilon} dx = 0,
    \ee and the Cauchy-Schwarz estimate
    \be 
\int_{\Omega} J_{\epsilon} f q^{\epsilon} dx = \int_{\Omega} \l^{-\fr{\alpha}{2}} J_{\epsilon}f \l^{\fr{\alpha}{2}}q^{\epsilon} dx 
\le \frac{1}{2}\|\l^{\fr{\alpha}{2}}q^{\epsilon} \|_{L^2}^2 + \frac{1}{2} \|\l^{-\fr{\alpha}{2}} J_{\epsilon}f \|_{L^2}^2.
    \ee Integrating \eqref{strong1} in time from $0$ to $t$ and taking the supremum over $[0,T]$ yields the uniform-in-$\epsilon$ regularity property
    \be 
q^{\epsilon} \in L^{\infty}(0,T; L^2) \cap L^2(0,T; \mathcal{D}(\l^{\fr{\alpha}{2}})). 
    \ee  
Now we multiply the first equation in $\eqref{epsreg}$ by $|q^{\epsilon}|^{\frac{1}{\delta} - 2} q^{\epsilon}$ and integrate over $\Omega$. We obtain 
\be 
\delta \fr{d}{dt} \|q^{\epsilon}\|_{L^{\fr{1}{\delta}}}^{\fr{1}{\delta}} + \int_{\Omega} |q^{\epsilon}|^{\frac{1}{\delta} - 2} q^{\epsilon}\l^{{\alpha}} q^{\epsilon} dx - \epsilon \int_{\Omega} |q^{\epsilon}|^{\frac{1}{\delta} - 2} q^{\epsilon} \Delta_{{D}} q^{\epsilon} dx = \int_{\Omega} |q^{\epsilon}|^{\frac{1}{\delta} - 2} q^{\epsilon} J_{\epsilon} f dx.
\ee By Theorem~\ref{thm:poin}, the nonlocal term $\int_{\Omega} |q^{\epsilon}|^{\frac{1}{\delta} - 2} q^{\epsilon}\l^{{\alpha}} q^{\epsilon} dx$ is nonnegative. Integrating by parts, we also obtain the nonnegativity of the regularization term $- \epsilon \int_{\Omega} |q^{\epsilon}|^{\frac{1}{\delta} - 2} q^{\epsilon} \Delta_{{D}} q^{\epsilon} dx$. For the forcing term, one has
\be 
\left|\int_{\Omega} |q^{\epsilon}|^{\frac{1}{\delta} - 2} q^{\epsilon}J_{\epsilon} f  dx \right|
\le C\|q^{\epsilon}\|_{L^{\fr{1}{\delta}}}^{\fr{1}{\delta}-1} \|J_{\epsilon}f\|_{L^{\frac{1}{\delta}}}
\ee thanks to H\"older's inequality. This gives the differential inequality 
\be 
\fr{d}{dt} \|q^{\epsilon}\|_{L^{\fr{1}{\delta}}} \le C\|J_{\epsilon}f\|_{L^{\fr{1}{\delta}}} \le C\|f\|_{L^{\fr{1}{\delta}}},
\ee after using the uniform-in-$\epsilon$ boundedness of the operator $J_{\epsilon}$ on $L^p$ spaces established in Lemma \ref{lpb}. As a consequence, the uniform-in-$\epsilon$ regularity 
\be 
q^{\epsilon} \in L^{\infty}(0,T; L^{\frac{1}{\delta}})
\ee follows. By passing to the limit $\epsilon\to0$, we conclude that a solution $q$ to the model \eqref{SQGmodel} also satisfies $q \in L^{\infty}(0,T; L^{\frac{1}{\delta}})$. 
{
Thanks to the regularity of $q$, we can also obtain that 
\begin{equation}\label{reg:qt}
    \partial_t q \in L^2(0,T; \mathcal D(\l^{-1-\frac\alpha2}).
\end{equation}
To see this, we only need to check the regularity of the nonlinear term $u\cdot \nabla q$. Indeed, consider any test function $\Phi\in L^2(0,T; \mathcal D(\l^{1+\frac\alpha2})$ and estimate
\begin{align*}
    \left|\int u\cdot \nabla q \Phi dx \right|= \left|- \int uq\cdot \na \Phi dx \right| \leq & \|q\|_{L^{\frac1\delta}} \|\Phi\|_{W^{1,\frac4{2-\alpha}}} \|q\|_{L^{\frac4{2-4\delta+\alpha}}} 
    \\
    \leq & C \|q\|_{L^{\frac1\delta}} \|\l^{1+\frac\alpha2} \Phi\|_{L^2} \|\l^{\frac\alpha2} q\|_{L^2}
\end{align*}
where we have used H\"older's inequality, the boundedness of Riesz transform on $L^{\frac1\delta}$, the embeddings $\mathcal D(\l^{1+\frac\alpha2}) \subset H^{1+\frac\alpha2} \subset W^{1,\frac4{2-\alpha}}$ and $\mathcal D(\l^{\frac\alpha2}) \subset H^{\frac\alpha2} \subset L^{\frac4{2-4\delta+\alpha}}$ that hold as $\delta<\frac{\alpha-1}2 < \frac\alpha2$. As $q$ satisfies \eqref{reg:q}, we conclude that $u\cdot \nabla q \in L^2(0,T; \mathcal D(\l^{-1-\frac\alpha2})$, and thus \eqref{reg:qt} follows.
}

As for uniqueness, suppose $q_1$ and $q_2$ are solutions to the model \eqref{SQGmodel} with the same initial data $q_1(0)= q_2(0)$ and homogeneous Dirichlet boundary conditions. We denote by $q$ and $u$ the differences $q = q_1 - q_2$ and $u = u_1 - u_2$. Then $q$ evolves according to 
\be \la{uniq1}
\pa_t q + \l^{\alpha}q = - u \cdot \na q_1 - u_2 \cdot \na q.
\ee 
Thanks to \eqref{reg:qt}, we can multiply this latter equation by $\l^{-1}q$ and integrate over $\Omega$ to obtain 
\be \label{uniq2}
\frac{1}{2} \frac{d}{dt} \|\l^{-\fr{1}{2}} q\|_{L^2}^2 + \|\l^{\frac{\alpha - 1}{2}} q\|_{L^2}^2 = - \int_{\Omega} u \cdot \na q_1 \l^{-1} q dx - \int_{\Omega} u_2 \cdot \na q \l^{-1} q dx. 
\ee The first nonlinear term on the right-hand side of \eqref{uniq2} vanishes as a consequence of the orthogonality property $R^{\perp} q \cdot Rq = 0$. In view of H\"older's inequality with exponents $\frac{2}{1-\delta}, \frac{2}{1-\delta}, \delta$, and the boundedness of the Riesz transform on $L^p$ spaces, it holds that 
\be 
\left|\int_{\Omega} u_2 \cdot \na q \l^{-1} q dx \right|
= \left|\int_{\Omega} u_2 Rq q dx \right| \le C\|u_2\|_{L^{\fr{1}{\delta}}} \|Rq\|_{L^{\frac{4}{2-2\delta}}} \|q\|_{L^{\fr{4}{2 - 2\delta}}}
\le C\|q_2\|_{L^{\fr{1}{\delta}}} \|q\|_{L^{\fr{4}{2 - 2\delta}}}^2.
\ee By the continuous embeddings of $\mathcal{D}(\l^{\delta})$ in $H^{\delta}$ and $H^{\delta}$ in $L^{\frac2{1-\delta}}$, the fact that the spaces $H_0^{s}$ and $\mathcal{D}(\l^s)$ are equivalent for $s \in (\fr{1}{2}, 1)$, and the Brezis-Mironescu fractional interpolation inequality, we have 
\begin{equation} 
\beg{aligned}
&\|q\|_{L^{\fr{4}{2-2\delta}}}^2
= \|\l^{\fr{1}{2}} (\l^{-\frac{1}{2}} q)\|_{L^{\fr{4}{2-2\delta}}}^2
\le C\|\l^{\fr{1}{2} + \delta} (\l^{-\frac{1}{2} } q)  \|_{L^2}^2
\le C\|\l^{-\frac{1}{2}} q\|_{H^{\fr{1}{2} + \delta} }^2
\\&\le C\|\l^{-\frac{1}{2} } q\|_{L^2}^{2\left(1 - \frac{1}{\alpha} - \frac{2\delta}{\alpha} \right)} \| \l^{-\frac{1}{2} } q\|_{H^{ \fr{\alpha}{2}}}^{2\left(\frac{1}{\alpha} + \frac{2\delta}{\alpha} \right)}
\le C\|\l^{-\fr{1}{2}} q\|_{L^2}^{2\left(1 - \frac{1}{\alpha} - \frac{2\delta}{\alpha} \right)} \|\l^{\fr{\alpha - 1}{2}} q\|_{L^2}^{2\left(\frac{1}{\alpha} + \frac{2\delta}{\alpha}  \right)}.
\end{aligned}
\end{equation} 
Note that $1-\frac1\alpha-\frac{2\delta}{\alpha}$ and $\frac1\alpha+\frac{2\delta}\alpha \in (0,1)$ since $\delta < \frac{\alpha-1}2$. 
By making use of Young's inequality for products with exponents $\frac{\alpha}{\alpha-1-2\delta}$ and $\frac{\alpha}{1+2\delta}$, we deduce that 
\be 
\frac{d}{dt} \|\l^{-\fr{1}{2}} q\|_{L^2}^2 \le C\|q_2\|_{L^{\fr{1}{\delta}}}^{\fr{\alpha}{\alpha -1 - 2\delta}} \|\l^{-\fr{1}{2}} q\|_{L^2}^2
\ee holds. An application of the Gronwall inequality yields the bound 
\be 
 \|\l^{-\fr{1}{2}} q(t)\|_{L^2}
 \le \|\l^{-\fr{1}{2}} q_0\|_{L^2} \exp \left\{C\int_{0}^{t} \|q_2(s)\|_{L^{\fr{1}{\delta}}}^{\fr{\alpha}{\alpha -1 - 2\delta}} ds   \right\}
\ee for all $t \in [0,T]$. As the initial data $q_0$ vanishes, we deduce that $q_1 = q_2$ for a.e. $x \in \Omega$ and $t \in [0,T]$.  
\end{proof}

\subsection{Properties of the Solution Map}

Let $\alpha \in (1,2)$. Fix a $\delta \in \left(0, \frac{\alpha - 1}{2}\right)$ and a time  $t \ge 0$ and  define the instantaneous  
solution map associated with the forced subcritical SQG equation 
\be 
\mathcal{S}_{\alpha}(t): L^{\fr{1}{\delta}}  \mapsto L^{\fr{1}{\delta}}  
\ee by 
\be 
\mathcal{S}_{\alpha}(t) q_0= q(t),
\ee where $q(t)$ is the unique solution of \eqref{SQGmodel} with initial datum $q_0$.

In this section, we investigate the properties of this solution map.

We start by proving the existence of a ball $\mathcal B_{\rho}$, compact in $H^{1}$, such that the image of $\mathcal B_{\rho}$ under $\mathcal{S}_{\alpha}(t)$ lies in $\mathcal B_{\rho}$ for large times. We need the following uniform Gronwall lemma:

\beg{lem} \cite{abdo2023long} \label{uniformgronwall} Let $y(t)$ be a nonnegative function of time $t$ that satisfies the differential inequality
\be  \la{gron1}
\fr{d}{dt} y + cy \le C_1 + C_2F_1 + C_3F_2y^n,
\ee where $c>0$ is a positive real number, $C_1, C_2$ and $C_3$ are nonnegative real numbers, $n$ is a nonnegative integer, and $F_1$ and $F_2$ are nonegative functions of time $t$. Suppose there exists a time $t_0$ and a positive number $R$ such that $y(t_0) < \infty$ and, for any $t \ge t_0$, it holds that 
\be \la{gron2}
\int_{t}^{t+1} F_1(s) ds \le R \quad \text{if $C_3 = 0$},
\ee 
\be \la{gron33}
\int_{t}^{t+1} \left[F_1(s) + F_2(s)y^{n-1}(s) + y(s) \right] ds \le R \quad \text{if $C_3\neq 0$ and $n \ge 1$.}
\ee Then there exists a positive constant $\rho=\rho(c,C_1, C_2, C_3, R)$ such that
for all times $t \ge t_0 + 1$, 
\be \la{gron4}
y(t) \le \rho.
\ee 
\end{lem}

\beg{prop} \la{absball} Let $\alpha\in(1,2)$ and fix some $s \in \left(1, \frac{\alpha + 1}{2}\right).$ 
Suppose $f \in \mathcal{D}(\l^{s- \frac{\alpha}{2}})$.  Then there exists a radius $\rho > 0$ depending only on the body forces and some universal constants such that for each $q_0 \in L^{\fr{1}{\delta}}$, there exists a time $T_0$ depending only on $\|q_0\|_{L^{\fr{1}{\delta}}}$, the body forces, and universal constants such that 
\be \la{absball1}
\mathcal{S}_{\alpha}(t) q_0 \in \mathcal{B}_{\rho} := \left\{q \in \mathcal{D}(\l^s): \|\l^s q\|_{L^2} \le \rho \right\}
\ee for all $t \ge T_0$. In particular, $\mathcal B_{\rho}$ is compact in $H^1$ and there is a time $T$ depending only on $f$ such that $\mathcal{S}_{\alpha}(t)\mathcal B_{\rho} \subset \mathcal B_{\rho}$ for all times $t\ge T$.
\end{prop}

\begin{proof} The proof is divided into several steps.

{\bf{Step 1. Evolution in $L^{\fr{1}{\delta}}$. }} The norm $\|q\|_{L^{\fr{1}{\delta}}}$ obeys the energy inequality
\be 
\fr{d}{dt} \|q\|_{L^{\fr{1}{\delta}}} +  c\|q\|_{L^{\fr{1}{\delta}}}
\le C\|f\|_{L^{\fr{1}{\delta}}},
\ee where $c$ is some constant depends on $\delta$. Here Theorem~\ref{thm:poin} is exploited.  By the uniform Gronwall Lemma \ref{uniformgronwall}, we deduce the existence of a radius $R_1$ depending only on $\|f\|_{L^{\fr{1}{\delta}}}$ and a time $t_1$ depending only $\|q_0\|_{L^{\fr{1}{\delta}}}$ such that the solution $q$ satisfies the uniform $L^{\fr{1}{\delta}}$ bound
\be \label{bound:1}
\|q(t)\|_{L^{\fr{1}{\delta}}} \le R_1
\ee for all $t \ge t_1$. 

{\bf{Step 2. Evolution in $L^{\infty}$.}} The $L^2 $ evolution of $q$, described by
\be 
\fr{1}{2} \fr{d}{dt} \|q\|_{L^2}^2 + \|\l^{\fr{\alpha}{2}}q \|_{L^2}^2 = \int_{\Omega} fq dx,
\ee boils down to 
\be \la{abs11}
\fr{d}{dt} \|q\|_{L^2}^2 + \|\l^{\fr{\alpha}{2}}q\|_{L^2}^2 \le \|\l^{-\fr{\alpha}{2}}f\|_{L^2}^2
\ee after applying Young's inequality. Bounding the dissipation from below using the embedding $\mathcal D(\l^{\frac\alpha2}) \subset L^2$, we obtain the decaying-in-time bound
\be 
\|q(t)\|_{L^2}^2 \le \|q_0\|_{L^2}^2 e^{-ct} + \|\l^{-\fr{\alpha}{2}}f\|_{L^2}^2
\ee for all $t \ge 0$, from which we deduce the existence of a time $t_2 \ge t_1$ depending only on $\|q_0\|_{L^2}$ such that 
\be 
\|q(t)\|_{L^2}^2 \le 1 + \|\l^{-\fr{\alpha}{2}}f\|_{L^2}^2
\ee for all $t \ge t_2$. Moreover, integrating \eqref{abs11} in time from $t$ to $t+1$ yields the local-in-time integrability estimates
\be \la{abs12}
\int_{t}^{t+1} \|\l^{\fr{\alpha}{2}}q(s)\|_{L^2}^2 ds \le 1+ 2\|\l^{-\fr{\alpha}{2}} f\|_{L^2}^2
\ee  for all times $t \ge t_2$. In particular, there exists a time $t_3 \ge t_2$ such that $\l^{\fr{\alpha}{2}}q (t_3)$ is square integrable. Now we address the time evolution of $\l^{\fr{\alpha}{2}}q$ in $L^2$ starting at time $t_3$.

The energy equality
\be 
\fr{1}{2}\fr{d}{dt} \|\l^{\fr{\alpha}{2}} q\|_{L^2}^2
+ \|\l^{\alpha} q\|_{L^2}^2 = \int_{\Omega} f \l^{\alpha} q dx - \int_{\Omega} u \cdot \na q \l^{\alpha} q dx
\ee holds and reduces to 
\be 
\fr{d}{dt} \|\l^{\fr{\alpha}{2}} q\|_{L^2}^2 + \fr{3}{2}\|\l^{\alpha} q\|_{L^2}^2 \le C\|f\|_{L^2}^2 + C\|u\|_{L^{\fr{1}{\delta}}} \|\l q\|_{L^{\frac{2}{1-2\delta}}} \|\l^{\alpha}q\|_{L^{2}}
\ee in view of H\"older and Young inequalities. By making use of the continuous embedding of $\mathcal{D}(\l^{ 2\delta})$ into $L^{\fr{2}{1-2\delta}}$ and the Brezis-Mironescu interpolation inequality, we estimate
\be 
\|\l q\|_{L^{\fr{2}{1-2\delta}}} \le C\|\l^{1 + 2\delta} q \|_{L^2} 
\le C\|\l^{\alpha} q\|_{L^2}^{\beta} \|\l^{\fr{\alpha}{2}}q\|_{L^2}^{1-\beta}
\ee for some $\beta \in (0,1)$, provided that $\fr{\alpha}{2} < 1 + 2\delta < \alpha$, which is equivalent to $\delta < \frac{\alpha-1}{2}$.  Thus, we obtain 
\be 
\fr{d}{dt} \|\l^{\fr{\alpha}{2}} q\|_{L^2}^2 + \|\l^{\alpha} q\|_{L^2}^2 \le C\| f\|_{L^2}^2 + C\|q\|_{L^{\fr{1}{\delta}}}^{\fr{2}{1-\beta}}  \|\l^{\fr{\alpha}{2}}q\|_{L^2}^2
\ee due to the boundedness of the Riesz transform on $L^{\fr{1}{\delta}}$ and Young's inequality. In view of the bound \eqref{bound:1} and local-in-time estimate \eqref{abs12}, we infer that the conditions of the Gronwall Lemma \ref{uniformgronwall} are satisfied in both cases (that is for any $\alpha \in (1,2)$). Consequently, there exists a time $t_4 \ge t_3$ and a radius $R_2$ depending only on $\|f\|_{L^{\fr{1}{\delta}}}$  such that 
\be 
\|\l^{\fr{\alpha}{2}} q(t)\|_{L^2} \le R_2
\ee for all $t \ge t_4$. Moreover, there is a radius $R_3$ depending also on $\|f\|_{L^{\fr{1}{\delta}}}$  such that 
\be \label{77}
\int_{t}^{t+1} \|\l^{\alpha} q\|_{L^2}^2 \le R_3
\ee for all $t \ge t_4$. In particular, there is a time $t_5$ at which $\l^{\alpha}q$ becomes square integrable. As $\mathcal{D}(\l^{\alpha})$ is continuously embedded in $L^{\infty}$, the solution $q$ is $L^{\infty}$ regular at time $t_5$.  
Since the $L^p$ norm of $q$ obeys \be  \fr{d}{dt} \|q\|_{L^p} + c\|q\|_{L^p} \le C\|f\|_{L^p} \ee for some positive constants $c, C$ being independent of $p$, we deduce that  
\be 
\|q(t)\|_{L^p} \le \|q(t_5)\|_{L^p} e^{-c(t-t_5)} + \frac{C}c\|f\|_{L^p}
\ee for all times $t \ge t_5$. Letting $p \rightarrow \infty$, it follows that 
\be  
\|q(t)\|_{L^{\infty}} \le \|q(t_5)\|_{L^{\infty}} e^{-c(t-t_5)} + \frac{C}c\|f\|_{L^{\infty}}
\ee for all $t \ge t_5$. Therefore, there exists a time $t_6 > t_5$ such that  
\be  \|q(t)\|_{L^{\infty}} \le 1 + \frac{C}c\|f\|_{L^{\infty}} \ee for all $t \ge t_6$.

{\bf{Step 3. Evolution in $\mathcal{D}(\l)$.}} From \eqref{77}, we infer the existence of a time $t_7 \ge t_6$ such that $\l q (t_7) \in L^2(\Omega)$.  We study the evolution of $\|\l q\|_{L^2}^2$ starting at time $t_7$. Indeed, we have 
\be 
\fr{1}{2} \fr{d}{dt} \|\l q\|_{L^2}^2
+ \|\l^{1+\fr{\alpha}{2}} q\|_{L^2}^2 
= \int_{\Omega} \l^{1-\fr{\alpha}{2}}f \l^{1+\fr{\alpha}{2}} q dx
+ \int_{\Omega} u \cdot \na q \Delta q dx. 
\ee Integrating by parts and applying the Brezis-Mironescu interpolation inequality, we estimate 
\be 
\beg{aligned}
\left| \int_{\Omega} u \cdot \na q \Delta q dx\right|
&\le C\int_{\Omega} |\na u| |\na q| |\na q| dx
\le C\|q\|_{W^{1,3}}^3
\le C\|q\|_{L^{\infty}} \|\l^{\fr{3}{2}}q\|_{L^2}^2 
\\&\le C\|q\|_{L^{\infty}} \|\l q\|_{L^2}^{\frac{2(\alpha - 1)}{\alpha}}\|\l^{1+\fr{\alpha}{2}}q\|_{L^2}^{\frac{2}{\alpha}}
\\&\le C\|q\|_{L^{\infty}} \|\l^{\alpha} q\|_{L^2}^{\frac{2(\alpha - 1)}{\alpha}}\|\l^{1+\fr{\alpha}{2}}q\|_{L^2}^{\frac{2}{\alpha}},
\end{aligned} 
\ee 
where the boundary temrs disappear since $u\cdot n = 0$ on the boundary.
By Young's inequality, the above gives
\be 
\left| \int_{\Omega} u \cdot \na q \Delta q dx\right|
\le \fr{1}{4} \|\l^{1+\fr{\alpha}{2}}q\|_{L^2}^2 + C \|q\|_{L^{\infty}}^{\frac{\alpha}{\alpha - 1}} \|\l^{\alpha}q\|_{L^2}^2 
\ee Therefore, we infer that 
\be 
\fr{d}{dt} \|\l q\|_{L^2}^2 + \|\l^{1+\fr{\alpha}{2}}q\|_{L^2}^2 
\le C\|q\|_{L^{\infty}}^{\fr{\alpha}{\alpha -1}}\|\l^{\alpha}q\|_{L^2}^2 + C\|\l^{1-\fr{\alpha}{2}}f\|_{L^2}^2.
\ee As a consequence of Step 2, the conditions of the uniform Gronwall Lemma \ref{uniformgronwall} hold for this latter differential inequality. Thus, there exists a time $t_8 \ge t_7$  depending only on the size of the initial data and the forcing term $f$ and a radius $R_4$ depending only on the forces $f$ such that 
\be 
\|\l q(t)\|_{L^2}^2 \le R_4
\ee and 
\be 
\int_{t}^{t+1} \|\l^{1+\fr{\alpha}{2}}q(s)\|_{L^2}^2 ds
\le R_4
\ee for all $t \ge t_8$. In particular, there exists a time $t_9 \ge t_8$ such that $\l^{1+\fr{\alpha}{2}}q(t_9) \in L^2$.

{\bf{Step 4. Evolution in $\mathcal{D}(\l^s)$.}} The $L^2$ norm of $\l^sq$ evolves according to 
\be 
\fr{1}{2} \fr{d}{dt} \|\l^s q\|_{L^2}^2 + \|\l^{s+\fr{\alpha}{2}}q\|_{L^2}^2
= - \int_{\Omega} u \cdot \na q \l^{2s} q dx + \int_{\Omega} \l^{s-\fr{\alpha}{2}}f \l^{s+\fr{\alpha}{2}} q dx,
\ee which, by the Cauchy-Schwarz inequality and the divergence-free condition obeyed by $u$, gives rise to 
\be 
\fr{1}{2} \fr{d}{dt} \|\l^s q\|_{L^2}^2 + \|\l^{s+\fr{\alpha}{2}}q\|_{L^2}^2
\le \|\l^{s-\fr{\alpha}{2}} (u \cdot \na q)\|_{L^2} \|\l^{s+\fr{\alpha}{2}}q\|_{L^2}
+ \|\l^{s-\fr{\alpha}{2}} f\|_{L^2} \|\l^{s+\fr{\alpha}{2}} q\|_{L^2}.
\ee 
In view of the continuous embedding of $H^{1+\frac{\alpha}{2}}$ into $C^{0, \frac{\alpha}{2}}$, the inequality  $\frac{\alpha}{2} > s - \frac{\alpha}{2}$, and the fact that $\na q \in \mathcal{D}(\l^{s-\frac{\alpha}{2}})$ when $s - \frac{\alpha}{2} < \frac{1}{2}$, the product estimate \eqref{fractionalproductt} applies and yields
\be 
\|\l^{s-\fr{\alpha}{2}} (u \cdot \na q)\|_{L^2}
\le C\|u\|_{L^{\infty}} \|\l^{s - \fr{\alpha}{2}} \na q\|_{L^2} + C\|\na q\|_{L^{2}} \|u\|_{H^{1+\fr{\alpha}{2}}}.
\ee Using the continuous Sobolev embeddings of $H^s$ into $L^{\infty}$, the boundedness of the Riesz transform from $\mathcal{D}(\l^s)$ into $H^s$, and the Brezis-Mironescu interpolation inequality, we estimate 
\be 
\|u\|_{L^{\infty}} 
= \|R^{\perp} q\|_{L^{\infty}}
\le C\|R^{\perp} q\|_{H^s}
\le C\|\l^s q\|_{L^2} \le C\|\l q\|_{L^2}^{\fr{2+\alpha -2s}{\alpha}} \|\l^{1+\fr{\alpha}{2}} q\|_{L^2}^{\fr{2(s-1)}{\alpha}}.
\ee Consequently, we obtain
\be 
\|\l^{s-\fr{\alpha}{2}} (u \cdot \na q)\|_{L^2}
\le C\|\l q\|_{L^2}^{\fr{2+\alpha -2s}{\alpha}} \|\l^{1+\fr{\alpha}{2}} q\|_{L^2}^{\fr{2(s-1)}{\alpha}}\|\l^{1+s - \fr{\alpha}{2}} q\|_{L^2} + C\|\l q\|_{L^2} \|\l^{1+\frac{\alpha}{2}}q\|_{L^2},
\ee 
where we have used the continuous embedding $\mathcal D(\l^\gamma) \subset H^\gamma$ that holds for all $\gamma\geq 0$.
Another application of the Brezis-Mironescu interpolation inequality gives
\be 
\|\l^{1+s - \fr{\alpha}{2}} q\|_{L^2}
\le C\|\l^s{q}\|_{L^2}^{\fr{2(\alpha - 1)}{\alpha}} \|\l^{s+\fr{\alpha}{2}}q\|_{L^2}^{\fr{2-\alpha}{\alpha}},
\ee so that 
\be 
\beg{aligned}
\|\l^{s-\fr{\alpha}{2}} (u\cdot \na q)\|_{L^2} \|\l^{s+\fr{\alpha}{2}}q\|_{L^2}
&\le C\|\l q\|_{L^2}^{\fr{2+\alpha -2s}{\alpha}} \|\l^{1+\fr{\alpha}{2}} q\|_{L^2}^{\fr{2(s-1)}{\alpha}} \|\l^s q\|_{L^2}^{\fr{2(\alpha - 1)}{\alpha}} \|\l^{s+\fr{\alpha}{2}}q\|_{L^2}^{\fr{2}{\alpha}}
\\&\quad\quad+ C\|\l q\|_{L^2} \|\l^{1+\frac{\alpha}{2}}q\|_{L^2}\|\l^{s+\frac{\alpha}{2}} q\|_{L^2}.
\end{aligned}
\ee  
By Young's inequality, we infer that 
\be 
\beg{aligned}
&\|\l^{s-\fr{\alpha}{2}} (u \cdot \na q)\|_{L^2} \| \l^{s+\fr{\alpha}{2}}q\|_{L^2}
\\
\le &\fr{1}{4}\|\l^{s+\fr{\alpha}{2}} q\|_{L^2}^2
+ C\|\l q\|_{L^2}^{\fr{2+\alpha - 2s}{\alpha - 1}} \|\l^{1+\fr{\alpha}{2}} q\|_{L^2}^{\fr{2(s-1)}{\alpha - 1}} \|\l^s q\|_{L^2}^2 + C\|\l q\|_{L^2}^2 \|\l^{1+\frac\alpha2} q\|_{L^2}^2
\\\le &\fr{1}{4}\|\l^{s+\fr{\alpha}{2}} q\|_{L^2}^2
+ C\left(\|\l q\|_{L^2}^{\fr{2+\alpha - 2s}{\alpha - s}} +  \|\l^{1+\fr{\alpha}{2}} q\|_{L^2}^{2} \right)\|\l^s q\|_{L^2}^2.
\end{aligned}
\ee Therefore, we end up with the differential inequality 
\be 
\fr{d}{dt} \|\l^s q\|_{L^2}^2
+ \|\l^{s+\fr{\alpha}{2}}q\|_{L^2}^2
\le C\|\l^{s-\fr{\alpha}{2}}f\|_{L^2}^2 +  C\left(\|\l q\|_{L^2}^{\fr{2+\alpha - 2s}{\alpha - s}} +  \|\l^{1+\fr{\alpha}{2}} q\|_{L^2}^{2} \right)\|\l^s q\|_{L^2}^2.
\ee As $s <\frac{\alpha+1}2 \le 1 + \fr{\alpha}{2}$, we have 
\be 
\int_{t}^{t+1} \|\l^s q(\tau)\|_{L^2}^2 d\tau \le R_4
\ee for any $t \ge t_9$ as a consequence of Step 3. Therefore,  the uniform Gronwall Lemma \ref{uniformgronwall} implies the existence of a time $t_{10} \ge t_9$ depending only on $\|q_0\|_{L^{\fr{1}{\delta}}}$ and the forcing term, and a radius $R_5$ depending only on $\|\l^{s-\fr{\alpha}{2}}f\|_{L^2}$ such that 
\be 
\|\l^{s} q(t)\|_{L^2}^2 \le R_5
\ee for all times $t \ge t_{10}$.
    
\end{proof}

{
\begin{rem}
    The condition $s<\frac{\alpha+1}2$ in Proposition~\ref{absball} is imposed in order to have the equivalence between $\mathcal D(\l^{s-\frac\alpha2})$ and $H^{s-\frac\alpha2}$ when $s-\frac\alpha2<\frac12$ as $\nabla q$ does not vanish on the boundary. Such an assumption can be relaxed to $s< 1+ \frac\alpha2$ and $s\neq \frac{\alpha+1}2$ by using part (1) of Theorem~\ref{thm:product-1} instead of part (3). The proof involves more tedious calculations and interpolation. For the sake of simplicity,  we make a stronger assumption on $s$ and use part (3) of Theorem~\ref{thm:product-1}.
\end{rem}
}

Next, we study the instantaneous continuity of the solution map $\mathcal{S}_{\alpha}(t)$ in $H^1$.

\beg{prop} Let $t > 0$ be a positive time. If $q_1^0, q_2^0 \in \mathcal{D}(\l)$, then 
\be\la{lip}
\|\l \mathcal{S}_{\alpha}(t)q_1^0 - \l \mathcal{S}_{\alpha}(t) q_2^0\|_{L^2}^2 \le K(t) \|\l q_1^0 - \l q_2^0\|_{L^2}^2,
\ee where 
\be 
K(t) = \exp \left\{C\int_{0}^{t} \left(\|\l^{2-\fr{\alpha}{2}}q_1(s)\|_{L^2}^2 + \|\l^{2-\fr{\alpha}{2}}q_2(s)\|_{L^2}^2 \right) ds \right\}.
\ee 
\end{prop}

\begin{proof}
     Let $q_1(t) = \mathcal{S}_{\alpha}(t)q_1^0$ and $q_2(t) = \mathcal{S}_{\alpha}(t)q_2^0 $. The difference $q(t) = q_1 -q_2$ obeys
    \be 
\fr{1}{2} \frac{d}{dt} \|\l q\|_{L^2}^2 + \|\l^{1+ \frac{\alpha}{2}}q\|_{L^2}^2
=  \int_{\Omega} u_1 \cdot \na q \Delta q dx + \int_{\Omega}  u \cdot \na q_2 \Delta q dx,
    \ee 
where $u=R^\perp q$. 
We estimate 
\be 
\beg{aligned}
\left|\int_{\Omega} u_1 \cdot \na q \Delta q dx  \right|
&\le C\int_{\Omega} |\na u_1| |\na q|^2 dx \le C\|q\|_{H^1} \|q\|_{W^{1,\frac{4}{2-\alpha}}} \|R^\perp q_1\|_{W^{1,\frac{4}{\alpha}}}
\\
&\leq C\|\l q\|_{L^2} \|q\|_{H^{1+\frac{\alpha}{2}}} \|R^{\perp} q_1\|_{H^{2-\frac{\alpha}{2}}}
\le C\|\l q\|_{L^2} \|\l^{1+\frac{\alpha}{2}}q\|_{L^2} \|\l^{2-\frac{\alpha}{2}} q_1\|_{L^2}
\end{aligned}
\ee and 
\be 
\beg{aligned}
\left|\int_{\Omega}  u \cdot \na q_2 \Delta q dx \right| 
&\le C\int_{\Omega} |\na u| |\na q_2| |\na q| dx \le C\|R^\perp q\|_{H^1} \|q\|_{W^{1,\frac{4}{2-\alpha}}} \| q_2\|_{W^{1,\frac{4}{\alpha}}}
\\&
\leq C \|\l q\|_{L^2} \|q\|_{H^{1+\frac{\alpha}{2}}} \|q_2\|_{H^{2-\frac{\alpha}{2}}}
\le C\|\l q\|_{L^2} \|\l^{1+\frac{\alpha}{2}}q\|_{L^2} \|\l^{2-\frac{\alpha}{2}} q_2\|_{L^2}
\end{aligned}
\ee   using the continuous embeddings of $H^{1+\fr{\alpha}{2}}$ and $H^{2-\fr{\alpha}{2}}$ into $W^{1,\frac{4}{2-\alpha}}$ and $W^{1,\frac{4}{\alpha}}$ respectively, and the boundedness of the Dirichlet Riesz transform from $\mathcal{D}(\l^{2-\frac{\alpha}{2}})$ into $H^{2 - \frac{\alpha}{2}}$ and $\mathcal{D}(\l)$ into $H^1$. 
An application of Young's inequality gives rise to the differential inequality
\be 
\fr{d}{dt} \|\l q\|_{L^2}^2 
\le C\left(\|\l^{2 - \frac{\alpha}{2}}q_1\|_{L^2}^2 + \|\l^{2 - \frac{\alpha}{2}} q_2\|_{L^2}^2 \right) \|\l q\|_{L^2}^2.
\ee By Gronwall's inequality, we obtain the desired Lipschitz continuity estimate \eqref{lip}. 
 
\end{proof}

Finally, we address the injectivity of the Solution map:

\beg{prop} Let $q_1^0, q_2^0 \in \mathcal{D}(\l)$. Suppose there is a time $T>0$ such that $S_{\alpha}(T)q_1^0 = S_{\alpha}(T)q_2^0$. Then $q_1^0 = q_2^0$. 
\end{prop}

\beg{proof} We adapt the proof of \cite{constantin1988navier} to the current system. Denote $S_{\alpha}(t)q_1^0$ and $S_{\alpha}(t)q_2^0$ by $q_1(t)$ and $q_2(t)$ respectively. 

{\bf{Step 1. Time analyticity of solutions.}} We complexify the functional spaces and operators, and fix an angle $\theta \in (-\frac{\pi}{2}, \frac{\pi}{2})$. We denote by $t$ the complex number $se^{i\theta}$ where $s > 0$. By the Chain Rule, we have 
\be 
\beg{aligned}
\frac{d}{ds} \|\na q(se^{i\theta})\|_{L^2}^2
&= \frac{d}{ds} (q(se^{i\theta}), -\Delta q (se^{i\theta}))_{L^2}
\\&= \left(e^{i\theta} \frac{dq}{dt} (se^{i\theta}), -\Delta q (se^{i\theta}) \right)_{L^2} + \left(q(se^{i\theta}), -e^{i\theta} \Delta \frac{dq}{dt} (se^{i\theta}) \right)_{L^2}
\\&= 2 Re \left(e^{i\theta} \left(\frac{dq}{dt} (se^{i\theta}), -\Delta q(se^{i\theta}) \right)_{L^2} \right),
\end{aligned}
\ee where $Re(z)$ is the real part of the complex number $z$. Consequently, it holds that
\be 
\frac{1}{2} \frac{d}{ds} \|\na q (se^{i\theta})\|_{L^2}^2
+ \cos \theta \|\l^{1+\frac{\alpha}{2}}q(se^{i\theta})\|_{L^2}^2
= Re \left(e^{i\theta} (-u \cdot \na q, -\Delta q)_{L^2} + e^{i\theta} (f, -\Delta q)_{L^2} \right).
\ee By integration by parts, we estimate the nonlinear term as 
\be 
|(- u \cdot \na q, -\Delta q)_{L^2}|
\le C\|\na u\|_{L^{\frac{4}{\alpha}}} \|\na q\|_{L^2} \|\na q\|_{L^{\frac{4}{2-\alpha}}}
\le C\|u\|_{H^{2-\frac{\alpha}{2}}} \|\l q\|_{L^2} \|\l^{1+\frac{\alpha}{2}} q\|_{L^2},
\ee where we have used the continuous Sobolev embeddings of $H^{2-\frac{\alpha}{2}}$ in $W^{1, \frac{4}{\alpha}}$ and $H^{1+\frac{\alpha}{2}}$ in $W^{1, \frac{4}{2-\alpha}}$. Since $1 < 2 - \frac{\alpha}{2} < 1 + \frac{\alpha}{2}$ and the Riesz transform is bounded from $\mathcal{D}(\l^{\gamma})$ into $H^{\gamma}$ for any $\gamma \ge 0$, we have 
\be 
\|u\|_{H^{2-\frac{\alpha}{2}}} \leq C \|\l^{2-\frac{\alpha}{2}} q\|_{L^2} \le C\|\l q\|_{L^2}^{1-\beta} \|\l^{1+\frac{\alpha}{2}}q\|_{L^2}^{\beta}
\ee for some $\beta \in (0,1)$. Thus, we obtain 
\be 
|(- u \cdot \na,q, -\Delta q)_{L^2}| \le C\|\l q\|_{L^2}^{2-\beta}\|\l^{1+\frac{\alpha}{2}}q\|_{L^2}^{1+\beta}
\le \frac{\cos \theta}{4} \|\l^{1+\frac{\alpha}{2}}q\|_{L^2}^2 + \frac{C}{(\cos \theta)^{\frac{1+\beta}{1-\beta}}} \|\l q\|_{L^2}^{\frac{2(2-\beta)}{1-\beta}},
\ee yielding the energy evolution 
\be 
\frac{d}{ds} \|\na q(se^{i\theta})\|_{L^2}^2 
+ \cos \theta \|\l^{1+\frac{\alpha}{2}} q(se^{i\theta})\|_{L^2}^2
\le \frac{C}{\cos \theta} \|\l^{1-\frac{\alpha}{2}}f\|_{L^2}^2 + \frac{C}{(\cos \theta)^{\frac{1+\beta}{1-\beta}}} \|\na q\|_{L^2}^{\frac{2(2-\beta)}{1-\beta}}.
\ee 
This gives rise to the bound 
\be 
\|\na q({se^{i\theta}})\|_{L^2}^2 
\le 2 \left(\|\na q_0\|_{L^2}^2 + 1 \right)
\ee provided that 
\be 
s \left(\frac{C}{\cos \theta} \|\l^{1-\frac{\alpha}{2}}f\|_{L^2}^2 + \frac{C}{(\cos \theta)^{\frac{1+\beta}{1-\beta}}}  \right) \le \Gamma_0
\ee where $\Gamma_0$ is a constant depending only on $\|\l q_0\|_{L^2}$. Therefore, $q$ is analytic on the region 
\be 
\mathcal{R} = \left\{t = s e^{i\theta} : s \left(\frac{C}{\cos \theta} \|\l^{1-\frac{\alpha}{2}}f\|_{L^2}^2 + \frac{C}{(\cos \theta)^{\frac{1+\beta}{1-\beta}}}  \right) \le \Gamma_0  \right\}.
\ee Due to the uniform boundedness of $q$ in $\mathcal{D}(\l)$, the analyticity of $q$ extends globally. 

{\bf{Step 2. Backward uniqueness.}} Since $q_1(T) = q_2(T)$, we have $q_1(t) = q_2(t)$ for all $t \ge T$ by the uniqueness of solutions. The time analyticity obtained in Step 1 yields $q_1(t) = q_2(t)$ for all $t \ge 0$. Therefore, $q_1^0 = q_2^0$. 

\end{proof}

\subsection{Existence of a Finite-Dimensional Global Attractor} As a consequence of the existence of a compact connected absorbing ball, and the continuity and injectivity of the solution map, we obtain the existence of a global attractor:

\beg{thm}\label{thm:sqg-3}
Let $\alpha \in (1, 2)$. There exist a time $T > 0$ depending only on the body forces $f$ and the power $\alpha$ such that the ball  $\mathcal{S}_{\alpha}(t) \mathcal{B}_\rho \subset  \mathcal{B}_\rho$ for all $t \ge T$, {where $\mathcal B_\rho$ is defined in \eqref{absball1}}. Moreover, the set
 \begin{equation} \la{X}
X = \bigcap_{t > 0} S_{\alpha}(t) \mathcal{B}_{\rho}
\end{equation}
satisfies the following properties: 
\begin{enumerate}
\item[(a)] $X$ is compact in $\mathcal{D}(\l)$.
\item[(b)] $S_{\alpha}(t)X = X$ for all $t \geq 0$.
\item[(c)] If $Z$ is bounded in $\mathcal{D}(\l)$  in the norm of $\mathcal{D}(\l)$, and $S_{\alpha}(t)Z = Z$ for all $t \geq 0$, then $Z \subset X$. 
\item[(d)] For every $w_0 \in \mathcal{D}(\l),$
$\lim\limits_{t \to \infty} dist_{\mathcal{D}(\l)} (S_{\alpha}(t)w_0, X) = 0$.
\item[(e)] $X$ is connected.
\end{enumerate}
\end{thm}

The proof is standard and follows from \cite{constantin1988navier}. We omit the details.

Now we study the dimensionality of the attractor. 
Fix $\alpha \in (1,2)$. For $N \ge 1$, we consider a smooth function $\Phi: \Omega \subset \RR^N \rightarrow \mathcal{D}(\l)$ and let $\Sigma_t$ be the image of $\Phi(\Omega)$ under the solution map $\mathcal{S}_{\alpha}(t)$ at time $t$. Denoting the volume element in $\RR^n$ by $d\omega_1 \dots d\omega_N$, the volume element in $\Sigma_t$ is given by 
\be 
\left|\frac{\partial}{\partial \omega_1} \mathcal{S}_{\alpha}(t) \Phi(\omega) \wedge \dots \wedge \frac{\partial}{\partial \omega_N} \mathcal{S}_{\alpha}(t) \Phi(\omega) \right| d\omega_1 \dots d\omega_N,
\ee where $\omega = (\omega_1, \dots, \omega_N) \in \RR^N$. The functions 
\be 
\xi_i = \frac{\partial}{\partial \omega_i} \mathcal{S}_{\alpha}(t) \Phi(\omega), \hspace{1cm} i = 1, \dots, N,
\ee solve 
\be 
\pa_t \xi_i + \l^{\alpha} \xi_i + R^{\perp} \bar{q} \cdot \na \xi_i + R^{\perp} \xi_i \cdot \na \bar{q} = 0
\ee along $\bar{q}(t) = \mathcal{S}_{\alpha}(t) \Phi(\omega)$. Let $\bar{q}_0 \in X$. We define the instantaneous volume
\be 
V_N(t) = \left\|\xi_1 \wedge \dots \wedge \xi_N \right\|_{\bigwedge^N \mathcal{D}(\l)}
\ee where $\xi_1, \dots, \xi_N$ solve along $\bar{q}(t) = \mathcal{S}_{\alpha}(t) \bar{q}_0$ and $\bigwedge^N \mathcal{D}(\l)$ is the $N$-th exterior product of $\mathcal{D}(\l)$.  

\beg{prop}
There is an integer $N_0$ depending only on $f$, and a positive constant $c$ depending only on $\alpha$ such that 
\be 
V_N(t) \le V_N(0) e^{-cN^{1+\frac{\alpha}{2}} t}
\ee for any $t \ge 0$ and $N \ge N_0$. 
\end{prop}

\beg{proof}
We denote by $I$ the identity operator and  consider the operators $\mathcal{A}_{\bar{q}_0}$ and $(\mathcal{A}_{\bar{q}_0})_{N}$ defined by 
\be 
\mathcal{A}_{\bar{q}_0}[\xi]
= \l^{\alpha} \xi + R^{\perp} \bar{q} \cdot \na \xi + R^{\perp} \xi \cdot \na \bar{q}
\ee and 
\be 
(\mathcal{A}_{\bar{q}_0})_{N} = \mathcal{A}_{\bar{q}_0} \wedge I \wedge \dots \wedge I + \dots + I \wedge \dots \wedge I \wedge \mathcal{A}_{\bar{q}_0}
\ee respectively. The wedge product $\xi_1 \wedge \dots \wedge \xi_N$ evolves in time according to 
\be 
\pa_t (\xi_1 \wedge \dots \wedge \xi_N) + (\mathcal{A}_{\bar{q}_0})_N (\xi_1 \wedge \dots \wedge \xi_N) = 0,
\ee which gives rise to the volume evolution equation 
\be 
\frac{d}{dt} V_N + Trace(\mathcal{A}_{\bar{q}_0}Q_N) V_N = 0,
\ee where $Q_N$ is the orthogonal projection  in $\mathcal{D}(\l)$ onto the space spanned by $\xi_1, \dots, \xi_N$. An application of Gronwall's inequality yields
\be 
V_N(t) \le V_N(0) \exp \left\{-\int_{0}^{t} Trace(\mathcal{A}_{\bar{q}_0} Q_N) ds \right\}
\ee for any $t \ge 0$. For each $t \ge 0$, we let $\left\{\phi_1, \dots, \phi_N \right\}$ be an orthonormal set spanning the linear span of $\left\{\xi_1, \dots, \xi_N \right\}$. Then 
\be 
\beg{aligned}
Trace(\mathcal{A}_{\bar{q}_0} Q_N) = \sum\limits_{i=1}^{N} (\l^{\alpha} \phi_i + R^{\perp} \bar{q} \cdot \na \phi_i + R^{\perp} \phi_i \cdot \na \bar{q}, -\Delta \phi_i)_{L^2}.
\end{aligned}
\ee Let $\mu_1, \dots, \mu_N$ be the first $N$ eigenvalues of $\l^{1+ \frac{\alpha}{2}}$. We have 
\be 
\sum\limits_{i=1}^{N} (\l^{\alpha} \phi_i, -\Delta \phi_i)_{L^2}  \ge \mu_1 + \dots + \mu_N \ge CN^{1+\frac{\alpha}{2}}.
\ee In view of the divergence-free condition obeyed by $R^{\perp} \bar{q}$, standard continuous Sobolev embeddings, and the boundedness of the Dirichlet Riesz transform from $\mathcal{D}(\l^s)$ into $H^s$, we  estimate
\be 
\beg{aligned}
\sum\limits_{i=1}^{N} |(R^{\perp} \bar{q} \cdot \na \phi_i, -\Delta \phi_i)_{L^2}| 
&\le C \sum\limits_{i=1}^{N} \|\na R^{\perp} \bar{q}\|_{L^{\fr{4}{\alpha}}} \|\na \phi_i\|_{L^2}\|\na \phi_i\|_{L^{\frac{4}{2-\alpha}}}
\\&\le C\|R^{\perp} \bar{q}\|_{H^{2-\frac{\alpha}{2}}} \|\na \phi_i\|_{L^2} \|\phi_i\|_{H^{1+\frac{\alpha}{2}}}
\\&\le C\|\l^{2-\frac{\alpha}{2}}  \bar{q}\|_{L^2} \|\na \phi_i\|_{L^2} \|\phi_i\|_{H^{1+\frac{\alpha}{2}}}.
\end{aligned}
\ee 
Since $\|\l \phi_i\|_{L^2} = 1$ and $2-\frac{\alpha}{2} < 1+\frac{\alpha}{2}$, this latter inequality reduces to 
\be 
\sum\limits_{i=1}^{N} |(R^{\perp} \bar{q} \cdot \na \phi_i, -\Delta \phi_i)_{L^2}| 
\le C\|\l^{1+\frac{\alpha}{2}} \bar{q}\|_{L^2} \|\l^{1+\frac{\alpha}{2}}\phi_i\|_{L^2}.
\ee By Theorem~\ref{thm:nonlinear}, we estimate 
\be 
\beg{aligned}
\sum\limits_{i=1}^{N} |(R^{\perp} \phi_i \cdot \na \bar{q}, -\Delta \phi_i)_{L^2}| 
&= \sum\limits_{i=1}^{N} |(\l^{1-\frac{\alpha}{2}}(R^{\perp} \phi_i \cdot \na \bar{q}), \l^{1+\frac{\alpha}{2}} \phi_i)_{L^2}|
\\&\le C\|R^{\perp} \phi_i\|_{H^1} \|\l^{1+\frac{\alpha}{2}} \phi_i\|_{L^2} \|\l^{1+\frac{\alpha}{2}} \bar{q}\|_{L^2}
\\&\le C\|\l \phi_i\|_{L^2} \|\l^{1+\frac{\alpha}{2}} \phi_i\|_{L^2} \|\l^{1+\frac{\alpha}{2}} \bar{q}\|_{L^2}
\\&\le C\|\l^{1+\frac{\alpha}{2}} \phi_i\|_{L^2} \|\l^{1+\frac{\alpha}{2}} \bar{q}\|_{L^2}.
\end{aligned}
\ee As $\bar{q}_0 \in X \subset \mathcal{B}_{\rho}$, we have
\be 
\int_{0}^{t} \|\l^{1+\frac{\alpha}{2}}\bar{q}\|_{L^2}^2 ds \le R_f t
\ee 
where $R_f$ is a constant depending only on $f$. Consequently, we infer that 
\be 
\beg{aligned}
&\int_{0}^{t} Trace(\mathcal{A}_{\bar{q}_0} Q_N)) ds 
\ge \int_{0}^{t} Trace(\l^{\alpha} Q_N) ds + \int_{0}^{t} (R^{\perp} \bar{q} \cdot \na \phi_i + R^{\perp} \phi_i \cdot \na \bar{q}, -\Delta \phi_i)_{L^2} ds
\\&\quad\quad\ge  \int_{0}^{t} Trace(\l^{\alpha} Q_N) ds - \frac{1}{2} \int_{0}^{t} Trace(\l^{\alpha} Q_N) ds  - CN \int_{0}^{t} \|\l^{1+\frac{\alpha}{2}}\bar{q}\|_{L^2}^2 ds 
\\&\quad\quad\ge \frac{1}{2} \int_{0}^{t} Trace(\l^{\alpha} Q_N) ds - CNR_ft
\\&\quad\quad\ge CN^{1+\frac{\alpha}{2}}t - CNR_ft
= CNt \left(N^{\frac{\alpha}{2}} - {R_f} \right)
\ge CNt
\end{aligned}
\ee provided that $N^{\frac{\alpha}{2}} \ge R_f + 1$.   
\end{proof}

As a consequence of the decay of volume elements, and following \cite{constantin1988navier}, one obtains:

\beg{thm}\label{thm:sqg-4}
The attractor $X$ has a finite fractal dimension in $\mathcal{D}(\l)$. That is, there exists a finite real number $\tilde M > 0$ depending on the body force $f$ such that 
\[ 
\limsup_{r\to 0}\fr{\log{N_{\mathcal{D}(\l)}(r)}}{\log\left(\fr{1}{r}\right)} \le \tilde M 
\] where $N_{\mathcal{D}(\l)}(r)$ is the minimal number of balls in $\mathcal{D}(\l)$ of radii $r$ needed to cover $X$.
\end{thm}

\section*{Acknowledgements}
Q.L. was partially supported by an AMS-Simons Travel Grant.

\vspace{0.5cm}

{\bf{Data Availability Statement.}} The research does not have any associated data.

\vspace{0.5cm}

{\bf{Conflict of Interest.}} The authors declare that they have no conflict of interest.

\bibliographystyle{plain}
\bibliography{reference}

\end{document}